\documentclass[12pt,twoside]{preprint}
\usepackage[margin=1.1in]{geometry}
\usepackage{amssymb}
\usepackage{amsmath}
\usepackage{hyperref}
\usepackage{breakurl}
\usepackage{mhenvs}
\usepackage{mhequ}
\usepackage{booktabs}
\usepackage{tikz} 
\usepackage{mathrsfs}
\usepackage{times}
\usepackage{microtype}
\usepackage{comment}
\usepackage{wasysym}
\usepackage{centernot}



\def\symbol#1{\textcolor{symbols}{#1}}
\def\1{\mathbf{\symbol{1}}}

\newtheorem{metatheorem}[lemma]{Meta-theorem}

\definecolor{darkred}{rgb}{0.9,0.1,0.1}

\newcommand{\eqdef}{\stackrel{\mbox{\tiny def}}{=}}

\newcommand{\E}{\mathbb{E}}
\newcommand{\N}{\mathbb{N}}
\renewcommand{\P}{{\mathbb P}}

\newcommand{\R}{\mathbb{R}}
\newcommand{\T}{\mathbb{T}}
\newcommand{\Z}{\mathbb{Z}}

\newcommand{\nn}{\mathfrak{m}}
\newcommand{\rr}{\mathsf{r}}
\newcommand{\KK}{\mathfrak{K}}

\newcommand{\Aa}{\mathcal{A}}

\newcommand{\Cc}{\mathcal{C}}
\newcommand{\Dd}{\mathcal{D}}

\newcommand{\Mm}{\mathcal{M}}
\newcommand{\Ss}{\mathcal{S}}

\newcommand{\MM}{\mathscr{M}}
\newcommand{\CN}{\mathcal{N}}

\newcommand{\al}{\nu}
\newcommand{\be}{\beta}
\newcommand{\ga}{\gamma}

\newcommand{\eps}{\varepsilon}
\newcommand{\ka}{\kappa}
\newcommand{\la}{\lambda}
\newcommand{\om}{\omega}

\newcommand{\si}{\sigma}

\newcommand{\ov}{\overline}

\newcommand{\msf}{\mathsf}
\newcommand{\Ll}{\left}
\newcommand{\Rr}{\right}
\newcommand{\Ca}{\Cc^{-\al}} 								

\newcommand{\ag}{\alpha}
\newcommand{\bg}{\be}
\newcommand{\dg}{\delta}
\newcommand{\eg}{\varepsilon}

\newcommand{\hg}{h_{\ga}} 								
\newcommand{\kg}{\ka_\ga} 								
\newcommand{\Kg}{K_\ga} 								
\newcommand{\LN}{\Lambda_N} 							
\newcommand{\Le}{\Lambda_\eps}		 					
\newcommand{\SN}{\Sigma_N} 							
\newcommand{\Hg}{\mathscr{H}_{\ga}} 						
\newcommand{\lbg}{\lambda_{\ga}} 							
\newcommand{\Zbg}{\mathscr{Z}_{\ga}} 						
\newcommand{\LgN}{\mathscr{L}_{\ga}} 						
\newcommand{\cg}{c_{\ga}}		 	 					
\newcommand{\Cg}{C_\ga} 								
\newcommand{\mg}{m_\ga} 								
\newcommand{\Mg}{M_\ga}  								
\newcommand{\Dg}{\Delta_\ga}		 					
\renewcommand{\ae}{\star_\eps} 							
\newcommand{\Eg}{E_\ga} 								
\newcommand{\Xg}{X_{\ga}}		  						
\newcommand{\co}{c_{\ga,2}}  								
\newcommand{\ct}{c_{\ga,1}}  								
\newcommand{\Xng}{X_\ga^0} 							
\newcommand{\Xn}{X^0} 							
\newcommand{\CGG}{\mathfrak{c}_{\ga}} 					

\newcommand{\Xe}{X_\eps}		 						
\newcommand{\Ze}{Z_\eps} 								
\newcommand{\Ce}{\mathfrak{c}_\eps} 						
\def\RR#1{R_{t}^{:{#1}:}} 									
\def\ZZ#1{Z^{:{#1}:}} 										

\def\Pg#1{P_{#1}^{\gamma}} 								
\def\hPg#1{\hat{P}_{#1}^{\gamma}} 							
\newcommand{\Zg}{Z_{\ga}}  								
\newcommand{\Rg}{R_{\ga,t}}		 						
\newcommand{\hRg}{\hat{R}_{\ga,t}}		 				
\def\RG#1{R_{\ga,t}^{:{#1}:}} 								
\def\hRG#1{\hat{R}_{\ga,t}^{:{#1}:}} 							
\def\ZG#1{Z_{\ga}^{:{#1}:}} 								

\newcommand{\Qg}{Q_{\ga,t}} 								
\def\EG#1{E_{\ga,t}^{:{#1}:}} 								

\newcommand{\taun}{\tau_{\ga,\nn}}							 
\def\RGn#1{R_{\ga,t,\nn}^{:{#1}:}}		 					
\newcommand{\Zgn}{Z_{\ga,\nn}} 							
\def\ZGn#1{Z_{\ga,\nn}^{:{#1}:}} 							
\newcommand{\sgn}{\sigma_{\ga,\nn}}		 				
\newcommand{\ssgn}{\sigma_{\ga,\nn}^{\prime}}		 		
\newcommand{\Mgn}{M_{\ga,\nn}} 							
\newcommand{\cgn}{c_{\ga,\nn}}							
\newcommand{\Cgn}{C_{\ga,\nn}}		 					
\newcommand{\bZg}{\mathbf{Z}_\ga} 						
\def\bRg#1{\mathbf{R}_{\ga,{#1} }} 							
\newcommand{\bZ}{\mathbf{Z}} 							
\newcommand{\bR}{\mathbf{R}_t}							

\newcommand{\ttZ}{ \widetilde{Z}_{\ga,\nn}} 					
\newcommand{\oX}{\overline{X}_{\ga,\nn}} 					
\newcommand{\tXgn}{X_{\ga,\nn}} 							
\newcommand{\tvg}{ v_{\gamma,\nn}} 						
\newcommand{\bvg}{ \overline{v}_{\gamma,\nn}} 				
\newcommand{\bZZg}{ \overline{Z}_{\gamma,\nn}} 				
\newcommand{\Psg}{ \Psi_{\gamma,\nn}} 					
\newcommand{\bPsg}{ \overline{\Psi}_{\gamma,\nn}} 			
\newcommand{\Ex}{\msf{Ext}} 								
\newcommand{\tCG}{\mathfrak{c}_{\ga}} 						
\def\ERR#1{\msf{Err}^{({#1})}} 								
\def\err#1{\msf{err}^{({#1})}} 								

\newcommand{\hKK}{\widehat{\KK}} 						
\newcommand{\hKg}{\hat{K}_{\ga}} 							

\newcommand{\dk}{\delta_k} 								

\begin{document}

\title{Glauber dynamics of 2D Kac-Blume-Capel model and their stochastic PDE limits}
\author{Hao Shen$^1$ and Hendrik Weber$^2$}
\institute{Columbia University, US, \email{pkushenhao@gmail.com} 
\and University of Warwick, UK, \email{hendrik.weber@warwick.ac.uk}}

\maketitle

\begin{abstract}
We study  the Glauber dynamics of a two dimensional Blume-Capel model (or dilute Ising model) with Kac potential parametrized by $(\beta,\theta)$ - the ``inverse temperature" and the ``chemical potential". We prove that the locally averaged spin field rescales to the solution of the dynamical $\Phi^4$ equation near a curve in the $(\beta,\theta)$ plane and to the solution of the dynamical $\Phi^6$ equation near one point  on this curve. Our proof relies on a discrete implementation of Da Prato-Debussche method \cite{dPD} as in \cite{MourratWeber} but an additional coupling argument is needed to show convergence of the linearized dynamics.
\end{abstract}


\setcounter{tocdepth}{2}
\tableofcontents

\section{Introduction}
\label{sec:intro}

The theory of singular stochastic partial differential equations (SPDEs)  has witnessed enormous progress in the last years. 
Most prominently,  Hairer's work on regularity structures 
\cite{Regularity}
allowed to develop  a stable notion of solution for a large class of SPDEs which
 satisfy a scaling condition called \emph{subcriticality}. 
Roughly speaking, a semi-linear SPDE equation is subcritical (or super-renormalizable), if the behaviour of solutions on small scales is dominated  
by the evolution of the linearized Gaussian dynamics.  The class of subcritical equations includes, for example, the KPZ equation in one spatial dimension, 
as well as  reaction diffusion equations with polynomial  nonlinearities 
\begin{equation}\label{e:RDE}
dX= \big(  \Delta X+ \sum_{k=1}^n \mathfrak{a}_{2k-1} X^{2k-1} \big)\, dt  +  dW  \qquad \mathfrak{a}_{2n-1}<0
\end{equation}
 driven by a space time white noise $dW$, if the space dimension $d$ satisfies $d< \frac{2n}{n-1}$ (of course strictly speaking the dimension $d$ has to be an integer but  one could emulate fractional dimensions
by adjusting the linear operator or the covariances of the noise).
 In particular, for $d =3$, equation \eqref{e:RDE} is only subcritical for the exponent $2n-1 = 3$ while for $d=2$, equation \eqref{e:RDE} is subcritical for all $n$. We will refer to these 
 equations as dynamical $\Phi^4_3$ and $\Phi^{2n}_2$ equations. Note that even in the subcritical case the expression \eqref{e:RDE}
 has to be interpreted with caution: for $d \geq 2$  a renormalization procedure which  amounts to subtracting one or several infinite terms 
 has to be performed. The fact that these solutions behave like the linearized dynamics on small scales but very nontrivially on large scales is related with the role they play in the description
of crossover regimes between universality classes in statistical physics. For example, the KPZ equation  describes the crossover regime between the Edwards-Wilkinson (Gaussian) fixed point and the ``KPZ fixed point",
while the dynamical $\Phi^4$ equation describes such a
crossover mechanism  between the Gaussian and the ``Wilson-Fisher fixed point". 
In two space dimensions the existence of infinitely many fixed points was predicted by conformal field theory, and the $\Phi^{2n}_2$ equations should describe the crossover regimes between the Gaussian and this family of fixed points (\cite[Fig.~4.3]{MR1219313}).


One  key interest when studying these SPDEs is to understand how they arise as scaling limits of various microscopic stochastic systems. Here it is important to note  that the equations are \emph{not}
scale invariant themselves (this is immediate from subcriticality). However,  they arise as scaling limits of systems with  tunable model parameters that are modified as the system is rescaled. 
Starting with  Bertini and Giacomin's famous  result \cite{BertiniGiacomin} on the convergence of the weakly asymmetric simple exclusion process to the KPZ equation,
by now many results in this direction have been obtained  for the KPZ equation (for example  \cite{MR2796514,DemboTsai,2015arXiv1505,CorwinShenTsai,labbe2016weakly} based on 
the Cole-Hopf transform,  \cite{MR3176353,gonccalves2016stochastic,diehl2016kardar}
based on the notion of energy solution, and \cite{KPZJeremy,CLTKPZ} based on regularity structures). Connections between the stationary $\Phi^4_2$ theory and Ising-like models were already observed in the seventies; early references 
include \cite{simonGriffiths} where the equilibrium $\Phi^4_2$ theory was obtained from an Ising-like model by a two-step limiting procedure.  The dynamical equation \eqref{e:RDE} in one dimension was obtained as a scaling limit for a dynamic Ising model with Kac interation  in the nineties \cite{MR1317994,MR1358083}. 
More precisely, the Kac Ising model is a spin model taking values in the $\{ \pm 1\}$ valued configurations over a graph ($\Z$ or a subinterval of $\Z$ in the case of  \cite{MR1317994,MR1358083}).
The static equilibrium model is given as the Gibbs measures associated to the Hamiltonian
\begin{equ} [e:Hamiltonian-Ising]
 \Hg(\si) = -  \frac12 \sum_{k, j }  \kg(k-j) \,\si(j) \,  \si(k),
 \end{equ}
 where $\kg$ is a non-negative interaction kernel parametrised by  $\gamma>0$ which  determines the
 interaction range between spins. In \cite{MR1317994,MR1358083} the Glauber dynamics for this model were considered and it was shown that
the locally averaged field $\hg = \sigma \ast \kg$ converges in law to a solution to the $\Phi^4_1$ equation when suitably rescaled.
Similar results in higher dimensions $d=2,3$ were conjectured in  \cite{MR1661764} but a complete proof in the two dimensional case was given only recently \cite{MourratWeber}. 
A similar convergence result is expected to hold in three dimensions, though a complete proof has not been established yet;  however in \cite{Phi4Weijun,ShenXu} it was shown that  a class of continuous phase coexistence models rescale to $\Phi^4_3$. 
\footnote{In \cite{Phi4Weijun} also different limits such as a dynamical $\Phi^3_3$ theory, which may blow up in finite time where obtained, but in order to achieve this the $\sigma \mapsto -\sigma$ symmetry in the model had to be broken.}


The tunable parameter in all of  the results on convergence of variants of the asymmetric simple exclusion process to KPZ, is the
asymmetry of the exclusion process:  making it smaller and smaller corresponds to making the model locally more ``Gaussian" which in turn 
corresponds to the fact that the dynamics on small scales  are dominated by solutions of the linear equation. In the Kac-Ising  case this tunable parameter is the range of the interaction kernel $\kg$.
As the system is observed on larger and larger scales locally more and more particles interact i.e. locally the system is closer to mean field.

 In order to obtain the scaling limit to $\Phi^4_2$ in \cite{MourratWeber}   five parameters had to be chosen in a certain way: three ``scaling parameters'' namely
 the space scaling, the time scaling, the rescaling of the field as well as two ``model parameters'', the range of the Kac interaction and the temperature. 
 It turns out that in order to obtain a non-linear scaling SPDE as scaling limit, one has to choose the  temperature close to the mean field critical value, although in two dimensions there is a small 
 shift  which corresponds to the renormalization procedure for the limiting equation, and a similar effect is expected in three dimensions. The remaining parameters have to be tuned in exactly the right way to balance all terms in the equation. It is natural to expect that in two space dimensions introducing additional parameters  should allow to balance even more terms leading to higher order terms in the equation. In this work we show that this is indeed the case. We allow for microscopic spin to take values in $\{\pm 1,0\}$ i.e. we add the possibility of a spin value $0$. The Hamiltonian thus becomes:
\begin{equ} [e:Hamilt0]
 \Hg(\si) = -  \frac12 \sum_{k, j }  \kg(k-j) \,\si(j) \,  \si(k) - \tilde\theta \sum_{j } \si(j)^2,
\end{equ}
where the extra parameter $\tilde \theta$ plays a role of chemical potential 
which describes a ratio of the number of ``magnetized" spins ($\sigma(j)\neq 0$)
over the number of ``neutral" spins ($\sigma(j)= 0$). In the limit $\tilde \theta \to \infty$
we recover the original Kac-Ising model.

 This model is the (Kac version) of the Blume-Capel model (initially proposed by \cite{Blume1966,Capel1966}). This Blume-Capel model  
 as well as the closely related (but slightly more complex) ``Blume-Emery-Griffiths" (BEG) model \cite{BlumeEmeryGriffiths}
 have been widely used to describe ``multi-critical" phenomena in equilibrium physics. Physicists also studied phase transitions for the Glauber type dynamics of mean field BEG model \cite{canko2006dynamic}.
Mathematically, the mean field model  in  equilibrium was  studied by in series of papers 
\cite{Ellis2005,Ellis2007,Ellis2010} (see more references therein), 
 analyzed the phase diagrams and proved that the suitably rescaled total spin 
converges to a random variable which is distributed with density $Ce^{-cx^2}$, $Ce^{-cx^4}$ or $Ce^{-cx^6}$ 
 in different regimes.
Also, the work \cite{BEGStein} obtained the rates of these convergences.
Regarding the dynamics, mixing theorems are also proved, see \cite{BCMixing,BEGMixing}.
The Blume-Capel model is also often referred as the (site) dilute Ising model (c.f. for instance the physics book \cite[Section~7.4.3]{francesco2012conformal} or on the mathematical side \cite{MR1797305,MR1341700} and references therein): one considers the site percolation of the square lattice with percolation probability $p$
 and the usual Ising model on the percolation clusters. The joint measure of the percolation and Ising model is then the Gibbs measure with Hamiltonian \eqref{e:Hamilt0} if one identifies $e^{\beta\tilde\theta}=(1-p)^{-1}-1$. The Glauber dynamics are then defined on both percolation and Ising configurations. The results of this article can then be stated as convergence  to the SPDEs by suitable tuning  the Ising temperature and percolation probability. 

 Our main result, Theorem~\ref{thm:Main}, shows that for a one parameter family of parameters we obtain the $\Phi^4_2$ equation in the scaling limit. 
This family ends at a ``tricritical point" where (after different rescaling) we get the $\Phi^6_2$ equation (see Figure~\ref{fig:1}). 
 Our equation for this curve of parameters and the value of the tricritical point coincide with the mean field results in  \cite{BlumeEmeryGriffiths},
 but as in the \cite{MourratWeber} logarithmic corrections to these mean field values are necessary to obtain the convergence results. These logarithmic 
 corrections correspond exactly to the ``logarithmic infinities'' that appear in the  renormalization procedures for the limiting equation.

\begin{figure} \label{fig:1}
\begin{center}
\begin{tikzpicture}[domain=0.25:4]
  \draw[very thin,color=gray] (-0.1,-0.1) grid (3.9,3.9);
  \draw[->] (-0.2,0) -- (4.2,0) node[right] {$a\eqdef e^{\beta \tilde \theta}$};
  \draw[->] (0,-0.2) -- (0,4.2) node[above] {$\beta$};
  \draw[color=blue]   plot (\x,{1+1/(2*\x)})    node[right] 
  	{\footnotesize critical curve $\mathcal C_c: \frac{2a}{2a+1} \beta -1=0$};
\node[circle,fill=red,inner sep=0pt,minimum size=2mm] at (0.25,3) {};
\node[above,red] at (1.2,3) {\footnotesize $(a_c^*,\beta_c^*)=(\frac14,3)$};
\end{tikzpicture}
\label{fig:1}
\caption{The Glauber dynamic of Blume-Capel model rescales to the $\Phi^4_2$ equation for  a curve of parameters in the $(\tilde \theta,\beta)$ plane, parametrized here in terms of 
$(a = e^{\beta \tilde \theta}, \beta)$. The leading coefficient  of the non-linearity in the limiting equation 
changes along the curve and vanishes at the tricritical point $(a_c^*,\beta_c^*)$. Close to this point a different rescaling leads to the $\Phi^6_2$ equation. Following the curve beyond this point 
would lead to a change of  sign in the leading order term resulting in finite time blowup of the corresponding SPDE.   }
\end{center}
\end{figure}
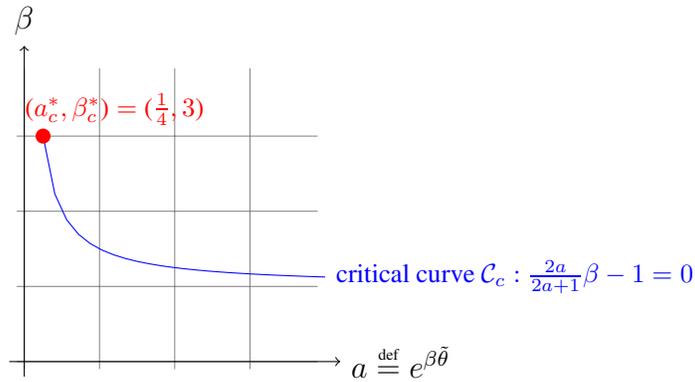

\begin{metatheorem}
Let $h_\gamma=\kg*\sigma$
be the locally averaged spin field of the Glauber dynamic of Kac-Blume-Capel model.
There exist a one parameter family of ``critical values" 
 and one ``tri-critical value", such that
when  $(\beta,\theta)$ approaches a critical value  at a suitable rate (which reflects the renormalization procedure for the limiting equation),
$X_\gamma(t,x)=\gamma^{-1} h_\gamma(t/\gamma^2,x/\gamma^2)$ converges to the solution
of the dynamical $\Phi^4$ equation, and when  $(\beta,\theta)$ approaches the tri-critical value at a suitable rate,
$X_\gamma(t,x)=\gamma^{-1} h_\gamma(t/\gamma^4,x/\gamma^3)$ converges to the solution
of the dynamical $\Phi^6$ equation.
\end{metatheorem}
It seems natural to conjecture that if one makes the model more complex (e.g. by allowing even more general spins and extra interaction terms in the Hamiltonian) any $\Phi^{2n}_2$ model
could be obtained.

On a technical level just as \cite{MourratWeber} our method relies on a discretization of Da Prato-Debussche's solution theory for \eqref{e:RDE} in two dimensions \cite{dPD}.
A main step is to  prove convergence in law  (with respect to the right topology) for the linearized dynamics as well as suitably defined ``Wick powers'' of these linearizations. 
In a second step this is then put into discretization of the ``remainder equation" and tools from harmonic analysis are used to control the error. The most striking difference in the present work
with respect to the technique in \cite{MourratWeber} is a difficulty to describe the fluctuation characteristics. In \cite{MourratWeber} the quadratic variation of the martingale 
$M_\gamma$ (see \eqref{e:evolution2} below for its definition) is equal to a deterministic constant up to a small error which can be controlled with a soft method. 
In the framework of the present paper this is not true anymore, and the quadratic variation has to be averaged over large temporal and spatial scales to characterize the 
noise in the limiting equation as  white noise. We implement this averaging by coupling the  spin field $\sigma(t,k)$ to a  much simpler field $\tilde \sigma(t,k)$ which can 
be analyzed directly. This auxiliary process
lacks the subtle large scale effects of $\sigma$ captured in our main result, but it has similar local jump dynamics and it turns out that $\sigma(t,k)$ coincides with $\tilde \sigma(t,k)$
for many $t$ and $k$ which is enough.

The structure of the paper is as follows. In Section 2 we discuss the two scaling regimes of our model  and formally derive the limiting equation in each regime. Section 3 is mainly aimed to show the convergence of the
linearized equation. It is here that we  present the coupling argument used to show the averaging  of the martingale fluctuation.
Section 4 contains the rest of the argument (the discrete Da Prato-Debussche method etc.). This part of the argument is close to \cite{MourratWeber},
but one difference with respect to  \cite{MourratWeber} is the replacement of the $L^\infty$ norm used there by an $L^p$ norm which becomes necessary because  of an  error term which arises in the coupling argument 
and which is only controlled in $L^p$.

\subsection*{Acknowledgements}
We would like to thank Weijun Xu for many helpful discussions on phase coexistence models and the dynamical $\Phi^4$ equations. H.S. was partially supported by the
NSF through DMS-1712684.

\section{Model, formal derivations and main result}
\label{sec:model-formal}

 The (Kac-)Blume-Capel model  in equilibrium is defined as a Gibbs measure $\lbg$ on the configuration space 
$\Sigma_N=\{-1,0,+1\}^{\Lambda_N}$ 
with  $\Lambda_N=\Z^2 /(2N+1)\Z^2$ being the two-dimensional discrete torus of size $2N+1$. More precisely
\begin{equ}
\lbg (\si) \eqdef \frac{1}{\Zbg}\exp\Big( - \be\Hg(\si) \Big)\; ,
\end{equ}
where $\beta>0$ is the inverse temperature, and $\Zbg$ denotes the normalization constant that is equal to the sum of the exponential weights over all configurations $\sigma\in\Sigma_N$.
The Hamiltonian $\Hg$ of the  model is defined via 
\begin{equ}\label{e:Hamiltonian}
\Hg(\si) \eqdef -  \frac12 \sum_{k, j \in \LN}  \kg(k-j) \,\si(j) \,  \si(k) 
- \tilde\theta \sum_{j \in \LN} \si(j)^2
\end{equ}
where $\tilde\theta$ is a real parameter, $\si\in\Sigma_N$, 
and $\kg$ is the interaction kernel which has  support size  $O(\gamma^{-1})$, which is constructed as follows:
Let $\KK \colon \R^2 \to [0,1]$ be a rotation invariant $\Cc^2$ function with support contained in  the ball of radius $3$ around the origin, such that
\begin{equation}
\label{e:norm-kk}
\int_{\R^2} \KK(x) \, dx = 1, \qquad \int_{\R^2} \KK(x) \, |x|^2 \, dx = 4 \;. 
\end{equation}
Then, for $0 < \gamma< \frac13$,   $\kg \colon \LN \to [0, \infty)$ is defined as $\kg(0) =0$ and
\begin{equation}\label{e:samplek}
\kg(k) = \frac{ \gamma^2 \, \KK(\ga k)  }{ \sum_{k \in \LN \setminus \{0\} }\ga^2 \, \KK(\ga k)}
 \qquad k \neq 0\;.
\end{equation}

We are interested in the following Glauber dynamics, a natural Markov process  on $(\Sigma_N,\lbg)$ which is reversible for $\lbg$. This process is defined in terms of the 
 jump rates $c_\gamma(\sigma; \sigma(j)\to \bar\sigma(j)) $ for a configuration $\sigma$, to change its spin  $\sigma(j)$ at position $j\in\Lambda_N$ 
to $\bar\sigma(j)\in\{\pm 1,0\}$.  This rate only depends on the final value $\bar\sigma(j)$ and is given by
\begin{equs}[2]
c_\gamma(\sigma,j,-1)&\eqdef
c_\gamma(\sigma; \sigma(j)\to -1) 
	= e^{-\beta h_\gamma(\sigma,j)+\theta} / \CN_{\beta,\theta}(h_\gamma(\sigma,j)) \;, \\
c_\gamma(\sigma,j,0)&\eqdef
c_\gamma(\sigma; \sigma(j)\to 0) = 1/\CN_{\beta,\theta}(h_\gamma(\sigma,j)) \;, \\
c_\gamma(\sigma,j,1)&\eqdef
c_\gamma(\sigma; \sigma(j)\to +1) 
	= e^{\beta h_\gamma(\sigma,j)+\theta}/\CN_{\beta,\theta}(h_\gamma(\sigma,j))
\end{equs}
where $\theta \eqdef \tilde\theta\beta$ and
$\hg$ is the locally averaged field
\begin{equ}\label{e:hg}
\hg(\si, k) \eqdef \sum_{j \in \LN } \kg(k-j) \, \si(j)=: \kg \star \si (k)\;,
\end{equ}
and $\CN_{\beta,\theta}(h_\gamma(\sigma,j))$ is a normalization factor
\begin{equ}
 \CN_{\beta,\theta}(h_\gamma(\sigma,j)) \eqdef
e^{-\beta h_\gamma(\sigma,j)+\theta} + 1 
	+ e^{\beta h_\gamma(\sigma,j)+\theta} \;.
\end{equ}
This can be written in a streamlined way
\begin{equ} \label{e:c-gamma}
c_\gamma(\sigma,j,\bar\sigma(j))
	= e^{ \bar\sigma(j) \beta h_\gamma(\sigma,j)+\bar\sigma(j)^2\theta} 
	/ \CN_{\beta,\theta}(h_\gamma(\sigma,j)) \;.
\end{equ}
The generator of the Markov process is then given by
\begin{equ}[e:GeneratorA]
\LgN f(\sigma) = \sum_{j\in\LN} \sum_{\bar\si(j)\in\{0,\pm 1\}} 
	\cg(\si,j,\bar\si(j)) \,(f(\bar\si)-f(\si))
\end{equ}
where $f:\SN\to \R$ and $\bar\si$ is the new spin configuration
obtained by flipping the spin $\si(j)$ in the configuration $\sigma$ to $\bar\si(j)$.
Let 
\begin{equ}
\hg(t,k) \eqdef \hg(\si(t),k)
\end{equ}
then one has
\begin{equ}\label{e:evolution1}
\hg(t,k) = \hg(0,k) + \int_0^t \LgN \, \hg(s, k) \, ds + \mg(t,k)\;,
\end{equ} 
where  the process $\mg(\cdot, k)$ is a martingale, whose explicit form (quadratic variation etc.) will be discussed in Section~\ref{sec:linear}. For the moment we focus on the drift 
term $\LgN \, \hg(s, k)$.
Since $\sigma$ and $\bar\sigma$ can only differ in their spin values at site $j$, one has
\begin{equ}
h_\gamma(\bar\si,k) - h_\gamma(\sigma,k)
	 = \kappa_\gamma(k-j) \, (\bar\sigma(j)-\si(j)) \;,
\end{equ}
and pluggin this into \eqref{e:GeneratorA} yields
\begin{equ}
\LgN  h_\gamma(\sigma,k)
=
\sum_{j\in\LN} 
\sum_{\bar\sigma(j)\in\{\pm1,0\}}
 \kg(j-k)\, (\bar\sigma(j)-\sigma(j)) \,c_\gamma(\sigma,j,\bar\sigma(j)) \;.
\end{equ}
Using the fact that $\sum_{\bar\sigma(j)\in\{\pm1,0\}} c_\gamma(\sigma,j,\bar\sigma(j)) =1$, one can alternatively write
\begin{equ}
\mathscr L_\gamma  h_\gamma(\sigma,k)
=
\sum_{j\in\LN} \kg(j-k) \Big( -\sigma(j)+
\sum_{\bar\sigma(j)\in\{\pm1,0\}}
  \bar\sigma(j) \,c_\gamma(\sigma,j,\bar\sigma(j))\Big)  \;.
\end{equ}
The Taylor expansion of $c_\gamma(\sigma,j,\bar\sigma(j))$ in $\beta \hg(\si,j)$ gives  
\begin{equ}[e:Taylor-c]
c_\gamma(\sigma,j,\bar\sigma(j)) = \sum_{n=0}^\infty c_n \, \beta^n \hg(\si,j)^n
\end{equ}
where the coefficients $c_n$ are given by (we only list the ones we will use):
\begin{equ}
c_1 = \frac{\bar\si(j) e^{\bar\si(j)^2\theta}}{1+2e^\theta} \;, \quad
c_3= \frac{\bar\si(j) e^{\bar\si(j)^2\theta} 
	\Big(\bar\si(j)^2+2 \left(\bar\si(j)^2 -3\right) e^\theta  \Big)}
	{6(1+2e^\theta)^2} \;,
\end{equ}
\begin{equ}
c_5= \frac{\bar\si(j) e^{\bar\si(j)^2\theta} 
	\Big( 4 \left(\bar\si(j)^2-5 \right)^2 e^{2\theta} 
		-2\left(8\bar\si(j)^2 +5 \right)e^\theta +\bar\si(j)^2 \Big)}
	{120(1+2e^\theta)^3} \;.
\end{equ}
Therefore one has
\begin{equs}
\LgN  h_\gamma(\sigma,k)
 &=
\Big(\kg \star \hg(\sigma, k) - \hg(\sigma,k) \Big)
	 + A_{\beta,\theta} \,\kg \star \hg(\sigma,k) \\
& \qquad +  B_{\beta,\theta} \,\kg \star \hg^3(\sigma,k)
+  C_{\beta,\theta} \,\kg \star \hg^5(\sigma,k)  + \dots
\end{equs}
where the remaining terms denoted by ``$\cdots$" are terms of the form $\kg \star \hg^{n}$ with $n$ odd and $n>5$, and
\begin{equs} [e:ABC]
A_{\beta,\theta} &\eqdef \frac{2a}{2a+1} \beta -1 \;,
\qquad
B_{\beta,\theta} \eqdef -\frac{a(4a-1)}{3(2a+1)^2} \beta^3  \;,\\
C_{\beta,\theta} &\eqdef \frac{a (64 a^2-26 a+1) }{60(1+2a)^3}\beta^5
\qquad
(a\eqdef e^\theta = e^{\beta\tilde\theta}) \;.
\end{equs}

Note that all the terms $\kg \star \hg^{n}$ with even powers $n$ vanish, because
$c_\gamma(\sigma,j,\bar\sigma(j))$ remains unchanged under 
$(\hg(\sigma,j),\bar\sigma(j)) \mapsto (-\hg(\sigma,j),-\bar\sigma(j))$,
thus the coefficients $c_n$ in \eqref{e:Taylor-c} for $n$ even must be 
even functions in $\bar\sigma(j)$.
 Multiplying this coefficient by $\bar\sigma(j)$ and 
summing over $\bar\sigma(j)\in\{\pm1,0\}$ necessarily yields zero.

\begin{remark}
As mentioned in Section~\ref{sec:intro}, letting $\theta\to \infty$ in the Hamiltonian \eqref{e:Hamiltonian} one recovers the Kac-Ising model.
Here in the above expansion for $\LgN  h_\gamma$,  if we send $\theta\to \infty$, we obtain the same coefficients in the corresponding expansion \cite[Eq.~(2.10)]{MourratWeber} for the Ising case.
\end{remark}

We set $\eg = \frac{2}{2N+1}$. Now every \emph{microscopic} point  $ k \in \LN$ can be identified with $x = \eg k \in \Le = \{x = (x_1,x_2) \in \eg \Z^2 \colon \,  x_1,x_2 \in(-1,1)  \}$. We view $\Le$ as a discretization of the 
continuous  torus $\T^2$ identified with $[-1,1]^2$.
We define the scaled field 
\begin{equ} \label{e:def-X}
X_\gamma(t,x)=\delta^{-1} h_\gamma(t/\alpha,x/\eps)\;,
\end{equ} 
so that
\begin{equs}
dX_\gamma(t,x) =  \Big(
&    \frac{ \eps^2}{\gamma^2 } \frac{1}{\alpha}  \widetilde{\Delta}_\gamma X_\gamma(t,x) 
+ \frac{A_{\beta,\theta}}{\alpha} K_\gamma \star_\eps X_\gamma(t,x)  
+ \frac{B_{\beta,\theta}\delta^2}{\alpha} K_\gamma \star_\eps X_\gamma^3  (t,x)  \\
&+ \frac{C_{\beta,\theta}\delta^4}{\alpha}  K_\gamma \star_\eps X_\gamma^5  (t,x)  
+ K_\gamma \star_\eps E_\gamma(t,x)  \Big)\,dt+ dM_\gamma(t,x) \;  ,
\label{e:evolution2}
\end{equs}
where the martingale $M_\gamma$ 
is defined by $M_\gamma(t,x)=\delta^{-1} m_\gamma(t/\alpha,x/\eps)$
and has an explicit quadratic variation
of order $\eps^2/(\delta^2\alpha)$ (see \eqref{Gauss2} below); 
the function
$\Kg(x) \eqdef \eg^{-2} \kg(\eg^{-1}x) $ 
is scaled to approximate the Dirac distribution;
the convolution $\ae$ on $\Le$ is defined through $X \ae Y(x) = \sum_{z \in \Le}\eg^2  X(x-z)  Y(z)$; and 
$\widetilde{\Delta}_\gamma X =\frac{\gamma^2}{\eg^2} (\Kg \ae X - X)$, so that $\widetilde{\Delta}_\gamma$ scales like the continuous Laplacian. 
The error term
$E_\gamma$ is given by
\begin{equ} [e:def-Eg]
E_\gamma = \frac{1}{\delta \alpha}
\Big(
	\frac{ \sum_{\bar\sigma \in\{\pm1,0\}} \bar\sigma \, e^{ \bar\sigma \beta \delta X_\gamma+\bar\sigma^2\theta} }
	{ \sum_{\bar\sigma \in\{\pm1,0\}}  e^{ \bar\sigma \beta \delta X_\gamma+\bar\sigma^2\theta} }
	- \frac{2a}{2a+1} \beta \delta X_\gamma 
	- B_{\beta,\theta}\delta^3   X_\gamma^3
	-C_{\beta,\theta}\delta^5   X_\gamma^5 \Big)
	\;.
\end{equ}

Now formally:
\begin{itemize}
\item
By choosing 
$A_{\beta,\theta} /\alpha =O(1)$
(which means that one tunes $\beta,\theta$ close to a curve in the $\beta-\theta$ plane given by $A_{\beta,\theta}=0$)
and the scaling of $\eps,\alpha,\delta$
such that the Laplacian, martingale and cubic terms are all of $O(1)$, namely
\begin{equ} \label{e:Phi4-scaling}
\eps \approx  \gamma^2, 
 \qquad 
 \alpha =\gamma^2,   
 \qquad 
 \delta =\gamma \;,
\end{equ}
one formally obtains the $\Phi^4$ equation, as long as 
$B_{\beta,\theta}\delta^2 / \alpha$ is strictly negative.
\item
However, if $(\beta,\theta)$ is tuned to be close to a special point 
$(\beta_c^*,\theta_c^*)=(3,-\ln 4)$ (which is a mean field value of a ``tricritical" point given by $A_{\beta,\theta}=B_{\beta,\theta}=0$) on the aforementioned curve, then under the scaling 
\eqref{e:Phi4-scaling}, the coefficient $B_{\beta,\theta}\delta^2 / \alpha$ 
vanishes, which would formally result in an Ornstein-Uhlenbeck process. To observe a nontrivial limit we have to consider a different scale.
In fact by imposing that both
$A_{\beta,\theta} /\alpha =O(1)$
 and $B_{\beta,\theta}\delta^2 / \alpha=O(1)$
and that
the Laplacian, martingale and quintic terms are all of $O(1)$,
namely
\begin{equ} \label{e:Phi6-scaling} 
\eps \approx \gamma^3,
 \qquad 
 \alpha = \gamma^4,  
 \qquad 
 \delta = \gamma \;,
\end{equ}
one formally obtains the $\Phi^6$ equation.
\end{itemize}
We will refer to the above two cases as ``the first (scaling) regime" and ``the second (scaling) regime". The curve in the $\beta-\theta$ plane was shown in Fig.~\ref{fig:1} 
Note that at $(\beta_c,\theta_c)$ the coefficient in front of $X^5$
is negative ($C_{\beta_c,\theta_c}=-9/20$) as desired for long time existence of solution.

Here, since the domain $\LN$ has integer size, 
we can only choose our space rescaling as
$ \eg = \frac{2}{2N+1}$, and
$N = \left\lfloor\gamma^{-2}\right\rfloor $ 
in the first regime or $N = \left\lfloor\gamma^{-3}\right\rfloor $ in the second regime. This is why we wrote $\approx$ above.
Write 
\begin{equ} [e:const-c-gamma]
\Delta_\gamma= c_{\gamma,2}^2 \widetilde{\Delta}_\gamma=\frac{ \eps^2}{\gamma^2 \alpha}  \widetilde{\Delta}_\gamma
\end{equ}
 where the coefficient $c_{\gamma,2}=\frac{ \eps}{\gamma^2 } $ in the first regime \eqref{e:Phi4-scaling} or $c_{\gamma,2}=\frac{ \eps}{\gamma^3 } $ in the second regime \eqref{e:Phi6-scaling}
and is close to $1$ up to an error $O(\gamma^2)$.



\begin{remark}
In $d$ space dimensions, the only difference in the above scaling arguments
is that the rescaled martingale $M_\gamma(t,x)$
has an explicit quadratic variation
of order $\eps^d/(\delta^2\alpha)$, so the condition of retaining 
Laplacian, martingale and quintic terms becomes
\begin{equ} 
\eps \approx \gamma^{\frac{6}{6-2d}} \;,
 \qquad 
 \alpha = \gamma^{\frac{2d}{3-d}} \;,  
 \qquad 
 \delta = \gamma^\frac{d}{6-2d} \;,
\end{equ}
It is manifest now that if $d=3$ the above relation cannot be satisfied,
which corresponds exactly to the fact that 
the subcriticality condition  for the $\Phi^6_d$ model is $d<3$.
This may be compared with the scaling for the $\Phi^4_d$ model 
in \cite[Remark~2.2]{MourratWeber} as following.
\begin{equ} 
\eps \approx \gamma^{\frac{4}{4-d}} , 
 \qquad 
 \alpha = \gamma^{\frac{2d}{4-d}},   
 \qquad 
 \delta = \gamma^{\frac{d}{4-d}} \;.
\end{equ}
\end{remark}

As discussed in \cite{MourratWeber}, the above formal derivation is not correct. Instead, in the first regime, fixing a point $(a_c,\beta_c)$ on the curve $\mathcal C_c$, one should write the linear and cubic terms as 
\begin{equ} \label{e:reorg-Hermite-1}
K_\gamma \star_\eps \Big(
 \frac{B_{\beta,\theta}\delta^2}{\alpha} \big(  X_\gamma^3 -3\CGG  X_\gamma \big)
+  \frac{A_{\beta,\theta}+3\CGG B_{\beta,\theta}\delta^2}{\alpha}   X_\gamma
\Big)
\end{equ}
where $\CGG$ is a logarithmically divergent renormalization constant, 
and tune $(a,\beta)$ such that $(A_{\beta,\theta}+3\CGG B_{\beta,\theta}\delta^2)/\alpha= \mathfrak a_1 + c_1(\gamma)$ where \ $\mathfrak a_1 \in \R$ is a fixed constant, and $c_1(\gamma)$ is a quantity vanishing as $\gamma\to 0$ which will give us certain freedom, namely,
\begin{equ} 
\frac{2a}{2a+1}  \bg -1 
= \gamma^2 \Big( \CGG\,\frac{a(4a-1)}{(2a+1)^2}\beta^3  +\mathfrak a_1+c_1(\gamma) \Big)\;.
\end{equ}
%
 The precise value of $\CGG$ will be  given below (Eq. \eqref{e:valueCGG}); the difference between $\beta_c \CGG$ and 
\begin{align*}
\sum_{\substack{\om \in \Z^2 \\ 0<| \om | < \ga^{-1}}} \frac{1}{4 \pi^2 |\om|^2}
\end{align*}
remains bounded as $\ga$ goes to $0$. 
One could well take $c_1(\gamma)=0$;
but the above tuning  is not very transparent because 
there are two parameters $(a,\beta)$ and the right hand side also involves $a,\beta$.
To make the tuning 
more explicit, 
  we can for instance first choose $a=a(\gamma)$ to be any sequence 
  such that $|a-a_c|=O(\gamma^2)$, 
and then replace the quantity  $\frac{a(4a-1)}{(2a+1)^2}\beta^3  $ 
by $\frac{a_c(4a_c-1)}{(2a_c+1)^2}\beta_c^3  $ with an error of $o(\gamma)$.
We then choose $c_1(\gamma)$ to exactly cancel this error,
and tune $\beta$ according to
\begin{equ}  \label{e:Phi4-beta-a}
\frac{2a}{2a+1}  \bg -1 
= \gamma^2 \Big( \frac{a_c(4a_c-1)}{(2a_c+1)^2}\beta_c^3  \,\CGG +\mathfrak a_1 \Big)\;,
\end{equ}
where $a$ stands for the sequence $a(\gamma)$ chosen above that converges to $a_c$.
Note that if $a\to \infty$ we recover from \eqref{e:Phi4-beta-a} 
the choice of $\beta$ in \cite[Eq~(2.18)]{MourratWeber}.

In the second regime, recall that the fifth Hermite polynomial is $x^5-10x^3+15x$. One should write the linear, cubic and quintic terms as
\begin{equs}
K_\gamma \star_\eps\Big(
\frac{A_{\beta,\theta}}{\alpha}  X_\gamma &
+ \frac{B_{\beta,\theta}\delta^2}{\alpha} X_\gamma^3
+ \frac{C_{\beta,\theta}\delta^4}{\alpha}  X_\gamma^5 \Big) \\
=K_\gamma \star_\eps\Big(
\frac{C_{\beta,\theta}\delta^4}{\alpha}  
	& \big(X_\gamma^5 -10\CGG X_\gamma^3 +15\CGG^2 X_\gamma \big)
+ \frac{B_{\beta,\theta}\delta^2 + 10 C_{\beta,\theta}\delta^4 \CGG}{\alpha}
	 \big(X_\gamma^3 -3\CGG X_\gamma \big) \\
& + \big( \frac{A_{\beta,\theta}}{\alpha}  
+ 3\CGG  \frac{B_{\beta,\theta}\delta^2 
	+ 5 C_{\beta,\theta}\delta^4 \CGG }{\alpha} \big)X_\gamma	
	\Big) \label{e:reorg-Hermite}
\end{equs}
So one should tune $(a,\beta)$ such that  the coefficient in front of 
$\big(X_\gamma^3 -3\CGG X_\gamma \big)$ is equal to $\mathfrak a_3 + c_3(\gamma)$ where \ $\mathfrak a_3 \in \R$ is a fixed constant; noting that $C_{\beta,\theta}=C_{\beta_c,\theta_c}+o(\gamma)=-9/20+o(\gamma)$, one can replace $C_{\beta,\theta}$ by $-9/20$ and suitably choose $c_3(\gamma) $ to cancel this error,
and thus obtain 
\begin{equ} \label{e:tune-tri-1}
-\frac{a(4a-1)}{3(2a+1)^2} \beta^3
 = \gamma^2 \Big( \frac{9}{2} \,\CGG +\mathfrak{a}_3 \Big)\;.
\end{equ}
One should furthermore impose that  the coefficient in front of 
$X_\gamma$ in \eqref{e:reorg-Hermite}
 is equal to $\mathfrak a_1 + c_1(\gamma)$ where \ $\mathfrak a_1 \in \R$ is a fixed constant, and suitably choose $c_1(\gamma) $ to get
\begin{equ}\label{e:tune-tri-2}
\frac{2a}{2a+1}  \bg -1 
= \gamma^4 \Big(-3\CGG \mathfrak{a}_3 - \frac{27}{4} \CGG^2+ \mathfrak{a}_1 \Big)\;.
\end{equ}
Combining the above two conditions, we can then obtain the correct tuning of the parameters $(\beta,a=e^\theta)$; we give their values in terms of power series in $\gamma$:
\begin{equs} [e:tune-tri]
a& =\frac14 - \gamma^2 \Big( \frac{9}{8} \,\CGG +\frac{\mathfrak{a}_3}{4} \Big) 
		+\frac{5}{48}\ga^4 \Big(81 \CGG^2 + 36 \CGG  \mathfrak{a}_3+ 4 \mathfrak{a}_3^2\Big)  +O(\ga^5)\;, \\
\beta&=3+ \ga^2 \Big(9 \CGG + 2 \mathfrak{a}_3 \Big)  
	+ \ga^4 \Big(-\frac{189}{4} \CGG^2 + 3 \mathfrak{a}_1 
			- 21 \CGG \mathfrak{a}_3 - \frac43 \mathfrak{a}_3^2 \Big) 
+O(\ga^5) \;.
\end{equs}
In fact, these precise values of $(a,\beta)$ do not matter in the sequel,
and it will be sufficient to know that there exists a family of $(a,\beta)$ depending on $\ga$ (approaching $(\frac14,3)$ as $\ga\to 0$) such that \eqref{e:tune-tri-1} and \eqref{e:tune-tri-2} do hold simultaneously.


\subsection*{The limiting SPDEs}
\label{sec:ContinuousAnalysis}

We briefly review the 
well-posedness theory for the $\Phi^{2n}$ equation 
\begin{equ} [e:Phi2nEqu]
dX= \big(  \Delta X+ \sum_{k=1}^n \mathfrak{a}_{2k-1} X^{2k-1} \big)\, dt  + \sqrt{2/\beta_c} dW
\qquad X(0)=X^0
\end{equ}
in two space dimensions 
with $\mathfrak{a}_{2n-1}<0$, and the parameter $\beta_c>0$ will correspond to a critical value of $\beta$ described above.
In order to interpret the solution to the above equation, let $W_\eps(t,x) = \frac{1}{4}\sum_{|\om| < \eps^{-1}} e^{i \pi \om \cdot x} \,\hat{W}(t,\om) $  be a spatially regularized cylindrical Wiener process, 
and consider the {\it renormalized} equation
 \begin{equ}\label{e:SPDEeps}
 d\Xe = \Big(  \Delta \Xe  + \sum_{k=1}^n \mathfrak{a}_{2k-1} H_{2k-1}(\Xe,\Ce) \Big)\, dt  + \sqrt{2/\beta_c}\, dW_\eps,
 \end{equ}
where $H_m = H_m(x,c)$ are Hermite polynomials defined recursively by setting 
$H_0 = 1$ and $H_{m} = x H_{m-1} - c \, \partial_x H_{m-1} $
so that $H_1 = x$, $H_2 = x^2-c$, $H_3 = x^3 - 3 cx$, etc. 
The constant $\Ce$ is given by
 \begin{equ} \label{e:norm-constant}
 \Ce = {\beta_c}^{-1} \!\!\! \sum_{ 0 <|\om| < \eps^{-1} } \frac{1}{ 4 \pi^2 |\om|^2} \;.
 \end{equ}
In particular,  the constants $\Ce$ diverge logarithmically as $\eps \to 0$. 
Then, \cite{dPD} shows that $\Xe$ converges to nontrivial limit.

More precisely, let 
\[
\Xe(t)= Z_\eps(t)+P_t X^0 + v_\eps(t)
\]
 where $P_t= e^{t \Delta }$ is the solution operator of the heat equation on the torus $\T^2$, and  
\begin{equ} 
 \Ze(t,\cdot) = \sqrt{2/\beta_c} \int_0^t P_{t-s} \, dW_\eps(s, \cdot) 
\end{equ} 
  is the solution to the linear equation with zero initial data.
Letting
\begin{equ}\label{e:def-Wick-eps}
Z_{\eps}^{:m:}(t,x) \eqdef H_m(\Ze(t,x), \Ce(t))
\end{equ}
for 
\begin{align}
\Ce(t)& =  \E[Z_\eps(t,0)^2]  = \frac{1}{2\beta_c}\sum_{ |\om| < \eps^{-1}} \int_0^t \exp\Ll(-2r\pi^2  |\om|^2\Rr) \, dr \;\notag\\
&= \frac{t}{2\beta_c} + \frac{1}{\beta_c}\sum_{0 < |\om| < \eps^{-1}} \frac{1}{4 \pi^2 |\om|^2 } \Big(1 - \exp\big( -2t \pi^2 \, |\om|^2  \big) \Big)   \;,\label{e:coft}
\end{align}
then $Z_{\eps}^{: m :}$  converge almost surely  and in every stochastic $L^p$ space with respect to the metric of $\Cc([0, T],\Ca)$ - this is essentially \cite[Lemma 3.2]{dPD}. We denote the limiting processes by  $Z^{: m:}$. 
 Note that $\Ce = \lim_{t \to \infty} ( \Ce(t) - \frac{t}{2\beta_c}) $, where the term $\frac{t}{2\beta_c}$ comes from the summand for $\om =0$ in \eqref{e:coft} which does not converge as $t \to \infty$. Furthermore, for every fixed $t>0$ the difference $|\Ce - \Ce(t)|$ is uniformly bounded in $\eps$.
This replacement of $\Ce$ by $\Ce(t)$ amounts to rewriting \eqref{e:SPDEeps} as \eqref{e:SPDEeps-t} below. 
Define 
$ \mathfrak{a}_{2k-1}^{(\eps)} (t)$ 
as time dependent coefficients such that
\begin{equ} [e:a-and-at]
\sum_{k=1}^n \mathfrak{a}_{2k-1} H_{2k-1}(x,\Ce)
=\sum_{k=1}^n \mathfrak{a}_{2k-1}^{(\eps)}(t) H_{2k-1}(x,\Ce(t)) \;.
\end{equ}
This is well-defined since the left hand side is an odd polynomial
of degree $2n-1$ which can be uniquely expressed as 
a linear combination of odd Hermite polynomials $H_{2k-1}(x,\Ce(t))$.
Note that the leading coefficients always satisfy 
$\mathfrak{a}_{2n-1}=\mathfrak{a}_{2n-1}^{(\eps)}(t) $. 
For the other coefficients, for instance, when $n=2$ one has $\mathfrak a_1^{(\eps)} (t)=3\mathfrak a_3 (\Ce(t)-\Ce)+\mathfrak a_1$; when $n=3$ one has
\begin{equs}  \label{e:a3t-a3}
\mathfrak a_3^{(\eps)} (t) &= 10 \mathfrak a_5 (\Ce(t)-\Ce)+\mathfrak a_3 \;,\\
\mathfrak a_1^{(\eps)} (t) &= -15 \mathfrak a_5 (\Ce(t)^2-\Ce^2)
	+3(\Ce(t)\mathfrak a_3^{(\eps)}(t) - \Ce \mathfrak a_3) +\mathfrak a_1 \;.
\end{equs}
In fact, plugging the first relation into the second, one has
\begin{equ} \label{e:a1t-a1}
\mathfrak a_1^{(\eps)} (t)  = 3\mathfrak a_3 (\Ce(t)-\Ce) + 15\mathfrak a_5 (\Ce(t)-\Ce)^2 + \mathfrak a_1 \;.
\end{equ}

Then \eqref{e:SPDEeps}  can be rewritten as
  \begin{equ}\label{e:SPDEeps-t}
 d\Xe = \Big(  \Delta \Xe  + 
 \sum_{k=1}^n \mathfrak{a}_{2k-1}^{(\eps)}(t) H_{2k-1}(\Xe,\Ce(t)) \Big)\, dt  + \sqrt{2/\beta_c}\, dW_\eps \;.
 \end{equ}
To proceed one needs the following simple fact, which generalizes \eqref{e:a1t-a1}.
\begin{lemma} \label{lem:ak-ake}
For every $k=1,\dots,n$,
the difference  $\mathfrak{a}_{2k-1} - \mathfrak{a}_{2k-1}^{(\eps)}(t) $ is a polynomial of $\Ce -\Ce(t)$ without zero order term, with coefficients only depending on $\mathfrak a_1,\cdots,\mathfrak a_{2n-1}$. This difference is uniformly bounded in $\eps$ for every $t>0$ and  diverges logarithmically in $t$ as $t\to 0$.
\end{lemma}
\begin{proof}
By the differential operator representation of  Hermite polymonials $H_m(x,c)=e^{-c\Delta/2} x^m$, where $\Delta$ is Laplacian in $x$ and the exponential is understood as power series without convergence problem when acting on polynomials. 
So we have 
\begin{equs}
H_{2k-1}(x,\Ce) &= e^{-\Ce \Delta /2} x^{2k-1}= e^{-\Ce(t)\Delta /2}e^{-(\Ce-\Ce(t))\Delta/2}  x^{2k-1} \\
&= e^{-\Ce(t)\Delta /2} H_{2k-1} (x,\Ce-\Ce(t)) \;.
\end{equs}
The operator $e^{-\Ce(t)\Delta /2}$ replaces every monomial term $x^m$ in the polymonial $H_{2k-1} (x,\Ce-\Ce(t))$ by $H_{m}(x,\Ce(t))$, which means that when re-expanding 
$H_{2k-1}(x,\Ce)$ on the left hand side of \eqref{e:a-and-at}  w.r.t. the basis $H_{m}(x,\Ce(t))$ the coefficients only depend on $\Ce,\Ce(t)$ via $\Ce-\Ce(t)$.
After this re-expansion we then compare the coefficients on the two sides of  \eqref{e:a-and-at}, noting that if $\Ce -\Ce(t)=0$ then $\mathfrak{a}_{2k-1}^{(\eps)}=\mathfrak{a}_{2k-1}$, and we obtain the first statement of the lemma. Note that
\begin{equ}  \label{e:ce-cet}
 \lim_{\eps \to 0} (\Ce -\Ce(t)) 
= - \frac{t}{2\beta_c}  + \sum_{\om \in \Z^2 \setminus\{ 0\}}  \frac{e^{- 2t\pi^2 |\om|^2}}{ 4\beta\pi^2 |\om|^2} \;.
\end{equ}
It is then obvious that the second statement of the lemma also holds. 
\end{proof}

By this lemma the limiting coefficient  $\lim_{\eps\to 0} \mathfrak{a}_{2k-1}^{(\eps)}(t) $ is integrable in $t$ at $t=0$.

As a convenient way to deal with the initial data $X^0$, we further define 
$\tilde Z (t) = Z(t)+P_t X^0 $ and
\begin{equ} \label{e:def-tilde-n}
\tilde Z^{: m :} (t) = \sum_{k=0}^m {m \choose k} 
    (P_t X^0)^{m-k}  Z^{: k :} (t) \, 
\end{equ}
The following theorem, essentially 
\cite[Theorem~6.1]{MourratWeberGlobal} (together with Remark~1.5 therein),
states that the equation
\begin{equ}\label{e:vequation-1}
\partial_t v = \Delta v 
+ \sum_{k=1}^{n} \mathfrak{a}_{2k-1}(t)
	 \sum_{\ell=1}^{2k-1} {2k-1 \choose \ell} \tilde Z^{: 2k-1-\ell :} v^\ell
\end{equ}
which is derived from \eqref{e:SPDEeps-t},
or equivalently 
\begin{equ}\label{e:vequation}
\partial_t v = \Delta v 
+ \sum_{\ell=1}^{2n-1}
	\Big( \sum_{k\in \Z\cap [\frac{\ell+1}{2}, n]}   \!\!\!\!\! \mathfrak{a}_{2k-1}(t)
	{2k-1 \choose \ell} \tilde Z^{: 2k-1-\ell :} \Big)v^\ell
\end{equ}
with zero initial condition $v(0)=0$ is globally well-posed.
The solution $v$ is the limit of $v_\eps$.

\begin{theorem}\label{thm:ContSolution}
For $\al>0$ small enough, fix an initial datum $\Xn \in \Ca$. 
For 
\[(Z, Z^{:2:}, \dots, Z^{:2n-1 :})  \in\big(  L^{\infty}\big([0,T], \Ca \big)\big)^{2n-1} \;,
\]
 let $(\tilde{Z},\widetilde{Z}^{:2 :},\cdots, \widetilde{Z}^{:2n-1 :}) $ be defined as in \eqref{e:def-tilde-n}. 
 Let $\Ss_T(Z, Z^{: 2: }, \dots, Z^{: 2n-1 : })$ denote the solution $v$ on $[0,T]$ of the PDE \eqref{e:vequation}. Then for any $\kappa >0$,
  the mapping 
  \[
  \Ss_T : \big(  L^{\infty}\big([0,T], \Ca \big)\big)^{2n-1} \to\Cc([0,T], \Cc^{2 - \al - \kappa}(\T^2) )
  \]
   is Lipschitz continuous on bounded sets .
\end{theorem} 

With the solution $v$ given by this theorem we call $X(t)= Z(t)+P_t X^0 + v(t)$ the solution to the dynamical $\Phi^{2n} $ equation 
\eqref{e:Phi2nEqu} with initial data $\Xn \in \Ca$.
(Due to the above theorem, Eq.~\eqref{e:Phi2nEqu} is sometimes written with each term $ X^{2k-1}$
 replaced by $:\!\!  X^{2k-1}\!\!:$ but we refrain from using this notation.)

\subsection*{Main result}

As in \cite{MourratWeber}, 
for any function $Y \colon \Le \to \R$, we define its smooth extension  to a function $\T^2 \to \R$ which is denoted by $\Ex Y$ (but sometimes still written as $Y$) in the following way:
\begin{equ} [e:Extension]
\Ex Y(x)  = \frac{1}{4} \sum_{\om \in \{-N, \ldots, N\}^2}   \sum_{y \in \Le}  \eg^2   \; e^{i \pi \om \cdot (x-y)} \; Y(y) 
 \qquad (x\in \T^2)
\end{equ}
which is the unique trigonometric polynomial of degree $\le N$ that coincides with $Y$ on $\Le$.

For any metric space $\Ss$, we denote by $\Dd(\R_+,\Ss)$ the space of $\Ss$ valued cadlag function endowed with the Skorokhod topology. For any $\al>0$ we denote by $\Ca$ the Besov space $B^{-\al}_{\infty, \infty}$ (see \cite[Appendix~A]{MourratWeber} for such spaces).

Assume that for $\ga>0$, the spin configuration at time $0$ is given by $\sigma_\gamma(0, k), \; k \in \LN$, and define for $x \in \Le$ 
\begin{align*}
\Xng(x) = \dg^{-1}\sum_{y \in \Le} \eg^2 \Kg(x -y) \,\sigma_\gamma(0, \eg^{-1}y ) \;.
\end{align*}
We smoothly extend $\Xng$ (in the way described above) to  $\T^2$ which is still denoted by $\Xng$. Let $\Xg(t,x), \, t \geq 0, \, x \in \Le^2$ be defined by \eqref{e:def-X} and extend $\Xg(t, \cdot)$ to $ \T^2$, still denoted by $\Xg$. 

Define 
\begin{equ} \label{e:valueCGG}
\CGG \eqdef 
\frac{1}{4\beta_c}
\sum_{\substack{\om \in \{-N, \ldots, N \}^2\\ \om \neq 0 }} 
	\frac{|\hKg(\om)|^2}{ \ga^{-b}  (1 - \hKg(\om))} \;,
\end{equ}
where $\hKg(\omega)=\sum_{x\in\Le } \eps^2 K_\ga(x)e^{-i\pi\omega\cdot x}$ is the Fourier transform of $K_\ga$,
$b=2$ in the first regime and $b=4$ in the second regime.

The main result of this article is the following.

\begin{theorem}\label{thm:Main}
Suppose that the precise value of $\CGG$ is given by 
\eqref{e:valueCGG}, and that $\Xng$ converges to $\Xn$ in $\Ca$ for $\al>0$ small enough and that  $\Xn, \, \Xng$ are uniformly bounded in  $ \Cc^{-\al +\ka}$ for an arbitrarily small $\ka>0$.

(1) Assume that the scaling exponents $\eps,\alpha,\delta$ satisfy \eqref{e:Phi4-scaling} and the parameters $a=e^\theta,\beta$ satisfy \eqref{e:Phi4-beta-a} for some $(a_c,\beta_c)$ and $\mathfrak a_1\in \R$ such that
\begin{equ} \label{e:critical-curve}
\frac{2a_c}{2a_c+1} \beta_c -1=0 \;.
\end{equ}
If $a_c > \frac14$,  then $\Xg$ converges in law to the solution of the following dynamical $\Phi^4$ equation:
\begin{equ}
dX= \big(  \Delta X+ \mathfrak a_1 X -\frac{a_c(4a_c-1) \beta_c^3 }{3(2a_c+1)^2}  X^3 \big)\, dt  + \sqrt{2/\beta_c} \,dW
\qquad X(0)=X^0 \;.
\end{equ}

(2) Under the same assumption in (1),
if $a_c = \frac14$,  then $\Xg$ converges in law to the linear equation:
\begin{equ}
dX= \big(  \Delta X+ \mathfrak a_1 X \big)\, dt  + \sqrt{2/3} \,dW
\qquad X(0)=X^0 \;.
\end{equ}

(3) Assume that the scaling exponents $\eps,\alpha,\delta$ satisfy \eqref{e:Phi6-scaling} and the parameters $a=e^\theta,\beta$ satisfy \eqref{e:tune-tri} for some $\mathfrak a_1, \mathfrak a_3 \in \R$ and in particular
\begin{equ} \label{e:tricritical-point}
(a,\beta) \to (1/4,3) \qquad \mbox{as } \gamma \to 0 \;.
\end{equ}
Then as $\gamma\to 0$,  $\Xg$ converges in law to the solution of a dynamical $\Phi^6$ equation:
\begin{equ}
dX= \big(  \Delta X+ \mathfrak a_1 X + \mathfrak a_3  X^3
	-\frac{9}{20 } X^5 \big)\, dt  + \sqrt{2/3} \,dW
\qquad X(0)=X^0 \;.
\end{equ}
All the above convergences are with respect to the topology of $\Dd(\R_+, \Ca)$. 
\end{theorem}

\begin{remark} \label{rem:2overbeta}
Note that the coefficient
$\sqrt{2/\beta_c}$ in front of the white noise in the limiting equations
makes the interpretation of  $\beta$ as ``inverse temperature"
more meaningful.
This means that the quadratic variation of our martingale 
should behaves like $2/\beta_c$ times the Dirac distribution.
The quadratic variation will depend on the spin configuration $\sigma$
and in the following proofs
we will approximate $\sigma$ by an i.i.d. spin system $\tilde \sigma$
so that at each site $\P(\tilde\sigma = \pm 1) = e^{\theta_c}/\mathcal N_c$ and $\P(\tilde\sigma = 0) = 1/\mathcal N_c$
where $\mathcal N_c=1+2e^{\theta_c}$.
(Recall that $\theta$ has the interpretation of ``chemical potential"
i.e. the ``ratio" between $\pm 1$ and $0$ spins.)
On average (over $\tilde\sigma\in\{-1,0,+1\}$) the quadratic variation
will then be shown as equal to (see \eqref{e:def-c-theta})
\[
 \frac{4e^{\theta_c}}{1+2 e^{\theta_c}} =\frac{2}{\beta_c}
 \]
where the last equality is by \eqref{e:critical-curve}
or \eqref{e:tricritical-point}.
\end{remark}

\begin{remark}
The limiting equations in the theorem are globally well-posed, see the paper
\cite{MourratWeberGlobal}, especially Remark~1.5 there. Actually, in case (1), if $a_c<\frac14$, one can still prove that $X_\ga$ converges to a $\Phi^4 $ equation, but with a plus sign in front of $X^3$, which may blow up in finite time.
\end{remark}


\section{Convergence of the linearized equation}
\label{sec:linear}

To prove the convergence result Theorem~\ref{thm:Main}
we rewrite our discrete evolution in the Duhamel's form:
\begin{equs}  [e:mildBEG6]
\Xg(t,\cdot) =&\Pg{t} \Xng + \int_0^t \Pg{t-s}  \Kg \star  
	 \Big( \frac{C_{\beta,\theta}\delta^4}{\alpha}\Xg^5  (s,\cdot) 
	 +\frac{B_{\beta,\theta}\delta^2}{\alpha}  \Xg^3  (s,\cdot) \\
+ & \frac{A_{\beta,\theta}}{\alpha}  \Xg(s,\cdot)
 + E_\gamma(s,\cdot) \Big) \, ds  + \int_{s=0}^t  \Pg{t-s} \, d\Mg (s,\cdot)  
 \qquad \mbox{on } \Lambda_\eps
\end{equs}
where the coefficients are defined in \eqref{e:ABC},
and $P_t^\gamma$ is the heat operator associated with $\Delta_\gamma$. Recall that the martingale $\mg$ was 
defined above in \eqref{e:evolution1} and the rescaled martingales $\Mg(t,z) = \frac{1}{\delta}\mg(\frac{t}{\alpha},\frac{z}{\eps})$ are defined on a rescaled grid $\Le \subseteq [-1,1]^2$.
%
An important step of proving convergence of \eqref{e:mildBEG6}
is to show convergence of the linearized system.
 For $x \in \Le$, we denote by 
\begin{equ} [e:Z-gamma]
\Zg(t,x) \eqdef \int_{r=0}^t  \Pg{t-r} \, d\Mg (r,x) 
\end{equ}
the stochastic convolution appearing as the last term  of \eqref{e:mildBEG6}. The process $\Zg$ is the solution to the linear stochastic equation
\begin{align}
d\Zg(t,x) &= \Dg \Zg(t,x) dt + d\Mg(t,x)\notag\\
 \Zg(0,x)& = 0 \, , \label{e:DefZg}
\end{align}
for $x\in \Le,\: t \geq 0$. 
As discussed in \eqref{e:Extension}, we extend $Z_\gamma$
to the entire torus $\T^2$ and still denote it by $Z_\gamma$.
The tightness of the family $\Zg$ with respect to the topology of  $\Dd(\R_+,\Ca) $ is established below in Prop.~\ref{prop:tight}. 
In this section we assume this result and  prove the convergence in law of $\Zg$   to the solution of the stochastic heat equation.

The predictable quadratic covariations  of the martingales $\mg(\cdot, k)$  are given by 
%
%
\begin{align}
&\langle \mg(\cdot, k) , \mg(\cdot, j) \rangle_t \notag \\
&= \int_0^t \sum_{\ell \in \LN} \kg (k - \ell) \kg (j - \ell)  \sum_{\bar{\sigma} \in \{ \pm 1, 0\}} (\bar{\sigma} - \sigma(s,\ell))^2 
\cg(\si(s), \ell, \bar{\sigma}) ds. \label{Gauss1}
\end{align}
Following the reasoning from \cite{MourratWeber} we first construct a modified version of the martingales $\Mg$ and the approximate stochastic convolution $\Zg$ for which we
have a  better control on this quadratic variation. 
To this end, we first  define the stopping time $\taun$ for a fixed  $\al \in (0,\frac12)$, any $\nn>1$ and $0 <\ga <1$, 
\begin{equation}
\label{e:deftaug}
\taun \eqdef \inf \big\{ t \ge 0 : \|\Xg(t,\cdot)\|_{\Ca} \ge \nn \big\} \;.
\end{equation}
For $k \in \LN$ and for $t \geq 0$, define
\begin{equation*}
\sgn(t,k) \eqdef 
\left\{
\begin{array}{ll}
\sigma(t,k)  & \text{if } t <   \frac{\taun}{\ag}, \\
\ssgn(t,k) & \text{otherwise}\;.
\end{array}
\right.
\end{equation*}
Here $\ssgn$ is a spin system with $\ssgn(\taun/\ag,k ) =\si(\taun^-/\ag,k )$, and  for every $t > \taun/\ag$ and every $k \in \LN$ the jumps to spin values $+1,0,-1$ at rates
 $\frac{e^{\theta_c}}{\CN_c},\frac{1}{\CN_c},\frac{e^{\theta_c}}{\CN_c}$ respectively, 
 independently from $\si$, with $\CN_c = 1 + 2 e^{\theta_c}$. 
 (Recall that $\theta_c$ is a critical value of $\theta$ as in Section~\ref{sec:model-formal}.)
In other words, the rate function $\cg$ is replaced by
\begin{equation}\label{e:raten}
\cgn^s (\sigma(s),k,\bar\sigma) =   
\left\{
\begin{array}{ll}
 \cg (\sigma(s),k,\bar\sigma)  & \text{if } s <   \frac{\taun}{\ag} \;, \\
 (\frac{e^{\theta_c}}{\CN_c},\frac{1}{\CN_c},\frac{e^{\theta_c}}{\CN_c}) & \text{otherwise}\;
\end{array}
\right.
\end{equation}
where in the second case, $\cgn^s (\sigma(s),k,\bar\sigma)$ is independent of the configuration $\sigma(s)$ and the site $k$ and thus only depends on $\bar\sigma$; so we only defined its values on the three points $\bar\sigma=1,0,-1$. 
 We now construct processes $\Mgn$ and $\Zgn$ following exactly the construction of $\Mg$ and $\Zg$ with $\si_\ga$ replaced by $\sgn$. 

Define the rescaled rate function
\begin{equ} [e:def-rateC]
\Cgn(s,z,\bar \sigma) \eqdef \cgn^{s/\ag}(\sgn(s/\ag), z/\eps,\bar \sigma) 
\end{equ}
for every $s\ge 0$, $z\in\Lambda_\eps$ and $\bar\sigma\in\{+1,0,-1\}$.
Of course $\Cgn(s,z,\bar \sigma)$ still depends on the configuration $\sgn$ but we suppress this dependence in the notation now.
For the  martingales $\Mgn(t,z)$, Eq. \eqref{Gauss1} turns into 
\begin{align}
&\langle \Mgn(\cdot, x ), \Mgn (\cdot, y) \rangle_t  \notag\\
& =\frac{\eps^2}{\delta^2\alpha} \int_0^t  \!\! \sum_{z \in \Le}\eg^2  \Kg( x -z) \, \Kg(y-z) 
\sum_{\bar{\sigma} \in \{ \pm 1,0\} } (\bar{\sigma} - \sigma(\alpha^{-1} s, \eps^{-1}z))^2 \Cgn\big(s,  z, \bar{\sigma}) \, ds. \label{Gauss2}
\end{align} 
Recall that the kernel $\Kg(x) = \eg^{-2} \kg(\eg^{-1}x) $ is a rescaled version of $\kg$ that behaves like an approximation of Dirac distribution $\delta$; thus we obtain $\eps^4$ when rescaling the two factors $\kappa_\ga$ but have moved an $\eps^2$ into the sum to anticipate that the sum over $z$ approximates $\delta(x-y)$, possibly times a constant.
Since $\delta=\ga$ in both ``scaling regimes", we can also write the coefficient in front of the integral as  $\co^2=\frac{\eps^2}{\ga^2\alpha}$ which was defined in \eqref{e:const-c-gamma}.
The constant $\co$ is close to $1$.
\begin{lemma}\label{Lemma3.1}
 The rates $\Cgn$ defined in \eqref{e:def-rateC} satisfy
\begin{align*}
\Cgn(s, z, \pm 1) &= \frac{e^{\theta_c}}{\CN_c} + E_\gamma \\
\Cgn(s, z, 0) &= \frac{1}{\CN_c} + E'_\gamma
\end{align*}
for every $s\ge 0$, $z\in\Lambda_\eps$,
where $\CN_c = 1 + 2 e^{\theta_c}$ and the random terms $E_\gamma, E'_\gamma$ which depend on $s,z$  are deterministically bounded by $C \gamma^{1 - 3 \nu}$
with constant $C$ depending linearly on $\nn$.
The un-rescaled rates $\cgn^s (\sigma(s),k,\bar\sigma)$ satisfy the same estimates for 
 every $s\ge 0$, $k\in\LN$ and $\bar\sigma\in\{\pm 1,0\}$.
\end{lemma}
\begin{proof}
By \eqref{e:def-rateC} it suffices to prove the stated estimates for $\Cgn$ and that for $\cgn$ immediately follow.
For $t > \taun$, we have $E_\gamma=E_\gamma'=0$ by definition.
For $t \le \taun$, first of all, we note that
$(\frac{e^{\theta}}{\CN},\frac{1}{\CN},\frac{e^{\theta}}{\CN})$ with $\CN = 1 + 2 e^{\theta}$ are nothing but the values
of $c_\ga$ defined in \eqref{e:c-gamma} for $\beta h_\ga=0$ at the three points $\bar\sigma=1,0,-1$.
Since the derivatives of the functions $\frac{x}{1+2x}$ and $\frac{1}{1+2x}$ are both bounded by $2$, the error caused by replacing 
$(\frac{e^{\theta}}{\CN},\frac{1}{\CN},\frac{e^{\theta}}{\CN})$
by
$(\frac{e^{\theta_c}}{\CN_c},\frac{1}{\CN_c},\frac{e^{\theta_c}}{\CN_c}) $
is bounded by $2|e^\theta-e^{\theta_c}| $; by the discussion above \eqref{e:Phi4-beta-a} (for the first scaling regime) or \eqref{e:tune-tri} (for the second scaling regime), this error is bounded by $C\gamma^{1 - 2 \nu}$.

Furthermore, it is easy to check by \eqref{e:c-gamma} that for any $\bar\sigma(j)\in\{\pm 1, 0\}$ and any $\theta\in \R$,
the rate $c_\gamma$ 
viewed as a function  of $\beta h_\gamma$ has the derivative:
\begin{equ} 
	 \frac{ 
		e^{ \bar\sigma(j) \beta h_\gamma+\bar\sigma(j)^2\theta}
		\Big( \bar\si(j) (e^{-\beta h_\ga+\theta} + 1 + e^{\beta h_\ga+\theta})
		 +e^{-\beta h_\ga+\theta}- e^{\beta h_\ga+\theta} \Big)} 
	{(e^{-\beta h_\ga+\theta} + 1 
	+ e^{\beta h_\ga+\theta})^2}
	 \;,
\end{equ}
which is
bounded by $2$. 
Therefore for $t < \taun$, 
\begin{equs}
[e:cg-control]
|E_\gamma| \vee |E_\gamma'|
&\le 2  \bg \big| \hg(\si(t/\ag),z/\eg)  \big| + C\gamma^{1-2\ka} =2 \bg \dg \big| \Xg(t,z) \big| + C\gamma^{1-2\ka} \\
&\le C(\al) \gamma^{1-3\al} (\|\Xg(t)\|_{\Ca}+1)\;.
\end{equs}
In the last step of \eqref{e:cg-control} we used the fact that
$\dg = \gamma$ in both scaling regimes; $\bg \le 4$ for $\gamma$ sufficiently small since in all three cases of Theorem~\ref{thm:Main} $\beta_c\le 3$; and the fact  that since the Fourier coefficients of $\Xg$ 
with frequency larger than $\gamma^{-2}$ (resp. $\gamma^{-3}$) vanish, by  \cite[Lemma~A.3]{MourratWeber},
$
\|\Xg(t)\|_{L^\infty} \le C \gamma^{-b\al} \|\Xg(t)\|_{\Ca}
$ with $b=2$ in the first regime (resp. $b=3$ in the second regime).
\end{proof}
This lemma allows to rewrite the last terms appearing in \eqref{Gauss2} as
\begin{align}\label{e:Gauss3}
\sum_{\bar{\sigma} \in \{ \pm 1,0\} } (\bar{\sigma} - \sigma(\alpha^{-1} s, \eps^{-1}z))^2 \Cgn\big(s,  z, \bar{\sigma}) = A(\sigma(\alpha^{-1} s, \eps^{-1}z)) + E_\gamma'',
\end{align}
where the error $E_\gamma''$ is again deterministically 
bounded by $C \gamma^{1 - 3 \nu}$ (for a constant  $C$ which depends on  $\nn$) and $A$ is a function defined on three points $\{+1,0,-1\}$ as following
\begin{align}\label{e:Gauss4}
A(\sigma)
=
\begin{cases}
2 e^{\theta_c}/\CN_c  \qquad  &\text{for } \sigma = 0\\
4 e^{\theta_c}/\CN_c + 1/\CN_c   \qquad  &\text {for } \sigma = \pm 1
\end{cases}
\end{align}
where $\CN_c = 1 + 2 e^{\theta_c}$ as before.
The main ingredient in the proof of  Theorem~\ref{t:converg-lin} below is to show that the dependence on the microscopic configuration $\sigma(t,x)$
in this expression becomes irrelevant when averaging over long time intervals,
and that $A$ may be replaced by its average.
%
%
%

Before stating Theorem~\ref{t:converg-lin}, we 
define a coupling between the microscopic spin process $\sigma(s, k)$ with an extremely simple  auxiliary
spin process $\tilde{\sigma}(s, k)$. For every given site $k\in\LN$ the spin  $\tilde{\sigma}(\cdot, k)$ gets updated 
at the same random times as the original process $\sigma(\cdot, k)$ but the update is determined according to a fixed 
probability distribution $\tilde{P}$ on $\{\pm 1, 0 \}$ independently of the values of both $\sigma$ and $\tilde{\sigma}$ and independently 
of other sites, which motivated by Lemma~\ref{Lemma3.1} is given by
\begin{align} \label{e:def-tilde-P}
\tilde{P} = 
\left(
\begin{array}{c}
 e^{\theta_c}/ \CN_c \\
 1/ \CN_c \\
 e^{\theta_c}/ \CN_c \\
\end{array}
\right).
\end{align}

This process $\tilde{\sigma}$ does not capture any of the subtle large scale non-linear effects of the field $\sigma$ described 
in our main result, but for any given site it coincides with $\sigma$ for many times which allows to replace $\sigma$ with $\tilde{\sigma}$ 
 below (see e.g. \eqref{e:Epppp}). 
The advantage of  this replacement is
that one can then average over $\tilde \sigma \in\{-1,0,+1\}$:
indeed,  note that by \eqref{e:critical-curve} and \eqref{e:tricritical-point} and the definition \eqref{e:Gauss4} for $A$ 
\begin{equ} \label{e:def-c-theta}
 \tilde \E A(\tilde\sigma(r,k))=\frac{e^{\theta_c}}{\CN_c} A(-1) 
	+ \frac{e^{\theta_c}}{\CN_c} A(1) +\frac{1}{\CN_c} A(0)
= \frac{4e^{\theta_c}}{1+2 e^{\theta_c}} =\frac{2}{\beta_c} \;,
\end{equ}
where $\tilde \E$ denotes the expectation with respect to $\tilde{P}$.
This is essentially the reason why
the pre-factor $\sqrt{2/\beta_c}$ in front of the noise of the limiting equation shows up (see Remark~\ref{rem:2overbeta}).
 In the proof of Theorem~\ref{t:converg-lin} we only make use of the averaging in time over $\tilde \sigma$. 
The proof of Proposition~\ref{prop:averageA} below then relies on the same construction and we will make use of the spatial averaging as well.

We now proceed to the construction of this coupling. By definition, for any fixed site $k \in \LN$ the process $\sigma(s, k)$ is a pure jump processes on $\{\pm 1, 0\}$. The joint law of all of these processes can be constructed as follows:
\begin{itemize}
\item For each site there is an independent Poisson clock, running at rate $1$.
\item At each jump of the Poisson clock the spin changes according to the transition probabilities given in the vector 
\begin{align*}
P(s,k) = 
\left(
\begin{array}{c}
\cgn^s(\sgn(s), k,1) \\
\cgn^s(\sgn(s), k,0) \\
\cgn^s(\sgn(s), k,-1) \\
\end{array}
\right).
\end{align*}
Of course this vector depends on the configuration of the neighboring particles at time $s$. 
\end{itemize}
The  transition probabilities of the auxiliary processes $\tilde{\sigma}(s, k), \; k \in \LN$ are fixed and given by \eqref{e:def-tilde-P}.
In order to construct the coupling, we note that 
according to Lemma~\ref{Lemma3.1} there exists a number $q$ satisfying
 \[
 1 \geq q \geq 1 -C \gamma^{1 - 3 \nu}\;,
 \]
such that  $q \tilde P \leq P$ where the inequality of the two vectors is to be understood entry by entry. Therefore, we can write
\begin{align*}
P(s,k)= q \tilde{P} + (1-q) R(s,k) ,
\end{align*}
where  $R$ is normalized to be a probability measure.
The coupling is now the following:
\begin{itemize}
\item At the initial time each of the $\tilde{\sigma}(0, k)$ is distributed according to $\tilde{P}$ and the realizations for different sites $k \neq k'$ are independent.
\item At each jump of the Poisson clock at site $k$, $\tilde{\sigma}(s, k)$ is updated according to $\tilde P$.
This update is independent from the updates at other sites as well as the 
jump times. 
\item To determine the updated spin for $\sigma(s, k)$ after the same jump of the Poisson clock, the vector $R(s,k)$ are evaluated.  It depends on the environment at the given time $s$. 
\item Toss a coin which yields $1$ with probability $q$ and $0$ with probability $1-q$. If the outcome of this toss is $1$ the spin $\sigma(s,k)$ is updated to the same value as $\tilde{\sigma}(s,k)$. 
 If the outcome is $0$ then $\sigma(s,k)$ is updated according to $R(s, k)$  independently of the update for $\tilde\sigma$. 
\end{itemize}
It is clear that the process $\tilde{\sigma}$ constructed in this way is a jump Markov chain jumping according to $\tilde{P}$ and that the processes for different sites are independent. This construction is consistent with the jumping rule of $\sigma$ (in particular $\sigma$ jumps  according to $P$).
Furthermore, for every $k\in\LN$, after each jump the probability that $\tilde{\sigma}(s,k) \neq \sigma(s,k) $ is bounded by $C \gamma^{1-3 \nu}$,
where the constant $C$ obtained from \eqref{e:cg-control} does not depend on the location $k$  and the jump-time.

To lighten the notation in the following calculation we introduce the centered random field $\bar A(\tilde\sigma(r,k)) = A(\tilde\sigma(r,k))-  \frac{2}{\beta_c}$ where $A$ was defined in \eqref{e:Gauss4}. 

\begin{lemma} \label{lem:sigtildeDecay}
For every $r, r'\ge 0$ and  $k,k'\in\LN$ we have
\begin{align*}
 \E  \bar A(\tilde\sigma(r,k))  \bar A(\tilde\sigma(r',k')) 
 	\le C \mathbf 1_{k=k'} e^{-|r-r'|}.
\end{align*}
\end{lemma}

\begin{proof}
Recall from the construction that for $k \neq k'$ the random variables $\tilde\sigma(r,k)$ and $\tilde\sigma(r',k')$ are independent and that therefore
for these $k \neq k'$ we have
\begin{align*}
 \E  \bar A(\tilde\sigma(r,k))  \bar A(\tilde\sigma(r',k')) =0.
\end{align*}
To get bounds in the temporal correlations for $\tilde\sigma(\cdot, k)$ for a fixed site $k$ we fix times $r'<r$ and denote by $\tau$ the 
first jump time of the Poisson clock for site $z$ after $r'$. Recall from the construction of $\tilde \sigma$ that  if $r < \tau$ the spin values of $\tilde \sigma(r, k)$ 
and $\tilde \sigma(r', k)$  are identical. The value after $\tau$ becomes independent of the value before $\tau$. 
With this discussion in mind we write
\begin{align*}
 &\E  \bar A(\tilde\sigma(r,k))\, \bar A(\tilde\sigma(r',k))\\
 &=   \E  \bar A(\tilde\sigma(r,k))^2 \mathbf{1}_{\tau > r} +  \E \bar A(\tilde\sigma(r,k))  \bar A(\tilde\sigma(r',k)) \mathbf{1}_{\tau \leq r } \;.
\end{align*}
The first term on the right hand side is bounded by
\begin{align*}
\E  \bar A(\tilde\sigma(r,k))^2 \mathbf{1}_{\tau >r}  \leq \sup_{\bar \sigma \in \{ \pm 1, 0 \}} |A (\bar \sigma)|^2\; \P (\tau >r) \leq  C e^{-|r-r'|}.
\end{align*}
For the second term we write
\begin{align*}
 &\E \bar A(\tilde\sigma(r,k))  \bar A(\tilde\sigma(r',k)) \mathbf{1}_{\tau \leq r } \\
 &= \E  \bar A(\tilde\sigma(r',k)) \mathbf{1}_{\tau \leq r }  \E\big( \bar A(\tilde\sigma(r,k))  \big| \mathcal{F}_\tau \big) =0,
\end{align*}
where $\mathcal{F}_\tau$ is the sigma algebra generated by $\tilde{\sigma}(\cdot, k)$ up to the stopping time $\tau$.
\end{proof}

\begin{theorem}[Convergence of $\Zg$]
\label{t:converg-lin}
Let $\al \in (0,1/2) $ and $\nn > 1$.  As $\gamma$ tends to $0$, the processes $\Zgn$ converge in law to $Z$ with respect to the Skorokhod topology on $\Dd(\R_+,\Ca)$, where $Z$ is defined as 
\begin{equ} 
 Z(t,\cdot) \eqdef \sqrt{2/\beta_c} \int_0^t P_{t-s} \, dW(s, \cdot) \;.
\end{equ} 
\end{theorem}

\begin{proof}
Proposition~\ref{prop:tight} below for the case $n=1$ shows that 
the family $\{ Z_{\ga,\nn}, \ga \in (0,\frac13) \}$ is tight on $\Dd(\R_+,\Ca) $ and any weak limit is supported on $\Cc(\R_+,\Ca) $. 
Given this tightness result,
we aim to show that any weak accumulation point $\bar{Z}$ solves the martingale problem discussed in Theorem 6.1 and Appendix C  of \cite{MourratWeber}.
The argument for the ``drift'' part of the martingale problem, namely establishing that 
$$
\MM_{\bar{Z},\phi}(t) \eqdef 
(\bar{Z}(t),\phi) - \int_0^t (\bar{Z}(s),\Delta \phi) \, ds\;
$$
is a local martingale for any test function $\phi \in \Cc^\infty$ is identical to \cite{MourratWeber}. Indeed, the claim we need to establish is that there exists a sequence 
of stopping times $T_n$ with $T_n \uparrow \infty$ a.s. as $n \to \infty$ such that for all $s<t$ and all  random variables $F$ which are bounded and measurable with respect to the 
 $\sigma$-algebra over $\Dd([0,s],\Cc^{-\al})$ we have
 \begin{equ}\label{e1}
\E \Big( \big(  \MM_{\bar{Z},\phi}(t \wedge T_n) -  \MM_{\bar{Z},\phi}(s \wedge T_n) \big) F \Big) = 0.
 \end{equ}
For any $\Cc^\infty$ function $\phi$
\begin{equ}\label{e2}
\Mm_{\ga,\phi}(t) = (\Zgn(t) , \phi)  -  \int_0^t  (\Zgn(s), \Dg \phi)  \, ds\;,
\end{equ}
is a martingale by assumption and therefore the formula \eqref{e1} with $\MM_{\bar{Z},\phi}$ replaced by $\Mm_{\ga,\phi}$ holds irrespective 
of the choice of stopping time $T_n$.
 Just as in \cite[Eq. (6.6)]{MourratWeber} it follows that the approximate Laplacian $\Dg$ appearing in expression  \eqref{e2}
 can be replaced by the full Laplacian $\Delta$ up to an error which is controlled by $C(\phi) \gamma^{2-2\ka}$ in both the ``first regime" and the 
 ``second regime''. By assumption the processes $\Zgn$ converge in law to $\bar{Z}$ and as the law of $\bar{Z}$ only charges the 
 space $\Cc(\R_+, \Cc^\nu)$,  in particular it assigns measure one to the set of continuity points (with respect to $\Dd(\R_+, \Cc^\nu)$
 topology) of the map that 
 sends $\bar{Z} $ to 
$\MM_{\bar{Z},\phi}(t)$
(recall that $\phi$ is smooth). Thus 
we can pass to the limit as soon as we have some control over the uniform integrability of these random variables. This is precisely the role of the 
stopping times - if we set $T_{L,\ga} =\inf\{t \ge 0 : \|\Zgn(t)\|_{\Ca} > L\} $ then it follows just as in \cite[Proof of Theorem~6.1]{MourratWeber} that (outside 
of a hypothetical countable set of values $L$) the processes $\Zgn(s\wedge T_{L,\gamma})$ 
also converge in law and furthermore for fixed $L,s,t$ the
random variables 
\begin{equ}\label{e3}
 (\Zgn(t \wedge T_{L,\gamma}) , \phi)  -  \int_0^{t\wedge T_{L,\ga}}  (\Zgn(s\wedge T_{L,\ga}), \Dg \phi)  \, ds\;,
\end{equ}
are uniformly bounded as $\gamma \to 0$ which permits to pass to the limit and establishes \eqref{e1}.

The more interesting part concerns the quadratic variation. More precisely, we need to show that 
$$
 \Ll(\MM_{\bar{Z},\phi}(t)\Rr)^2 - \frac{2t}{\beta_c} \|\phi\|_{L^2}^2\;
$$
is a local martingale; recall that the factor $2/\beta_c$ naturally appears from \eqref{e:def-c-theta}.

This follows if we can establish that for any fixed trigonometric polynomial  $\phi$. If we fix such a $\phi$, then as soon as $\gamma$ is small enough to guarantee the degree of $\phi$ is $ \le \gamma^{-2}$ (or $\ga^{-3}$ depending on the regime),  the quantity
\[
 (\Mgn(t), \phi)  =\sum_{x \in \Le}\eg^2 \Mgn(t,x) \phi(x)
\]
can be written using Parseval's identity (see    \cite[Appendix A]{MourratWeber})
\begin{align*}
\langle (\Mgn(t), \phi) \rangle 
&= \co^2 \sum_{x ,y\in \Le} \eps^4 \phi(x) \phi(y)\sum_{z \in \Le}\eg^2  \Kg( x -z) \, \Kg(y-z)  \\
& \quad  \times \int_0^t  \!\! \sum_{\bar{\sigma} \in \{ \pm 1,0\} } 
	(\bar{\sigma} - \sigma(\alpha^{-1} s, \eps^{-1}z))^2 \Cgn\big(s,  z, \bar{\sigma}) \, ds\\
&=    \frac{2t}{\beta_c} \|\phi\|_{L^2}^2 + E_\gamma'''(t),
\end{align*} 
for an error $E_\gamma'''(t)$ for which $\E |E_\gamma^{'''}(t)|  \to 0$ as $\gamma \to 0$. 
For this statement in turn \eqref{e:Gauss3} and \eqref{e:Gauss4} show that 
it is sufficient to prove that for every $z \in \Le$ we have
\begin{equ} \label{e:Epppp}
\int_0^t A(\sigma(\alpha^{-1} s, \eps^{-1}z  )) ds =  \frac{2t}{\beta_c}  + E_\gamma^{''''},
\end{equ}
with a good control on $E_\gamma^{''''}$. 
Indeed, one has $|\co^2 -1| \le O(\gamma^2)$ and by \eqref{e:norm-kk}, \eqref{e:samplek} and  $\Kg(x) = \eg^{-2} \kg(\eg^{-1}x) $,
\begin{equ}
\sum_{x ,y\in \Le} \eps^4 \phi(x) \phi(y)\sum_{z \in \Le}\eg^2  \Kg( x -z) \, \Kg(y-z)
\to  \|\phi\|_{L^2}^2 \;,
\end{equ}
independently of the scaling relation between $\eps$ and $\ga$ (thus it holds for both scaling regimes).
Although we have assumed that $\phi$ is a trigonometric polynomial, by \cite[Remark~C.4]{MourratWeber}, this is sufficient to characterize the law of $\ov{Z}$.

While the error terms $E_\gamma, E_\gamma', E_\gamma''$ were all deterministically bounded, we will only get a probabilistic bound 
for $E_\gamma'''$. To obtain this bound we will 
need the coupling between the microscopic spin processes $\sigma$ and $\tilde{\sigma}$.
Recall that for every $z$, after each jump the probability 
that $\tilde{\sigma}(\alpha^{-1}s,\eps^{-1} z) \neq \sigma(\alpha^{-1}s,\eps^{-1} z) $ is bounded by $C \gamma^{1-3 \nu}$,
where the constant $C$ does not depend on  $z$  and the jump-time.
We then get
\begin{align*}
\int_0^t A(\sigma(\alpha^{-1}s,\eps^{-1} z)) ds -  \frac{2t}{\beta_c}  =
& \int_0^t A(\tilde{\sigma} (\alpha^{-1}s,\eps^{-1} z)) ds -\frac{2t}{\beta_c}  \\
&+ \int_0^t A(\sigma (\alpha^{-1}s,\eps^{-1} z)) ds -  A(\tilde{\sigma} (\alpha^{-1}s,\eps^{-1} z)) \, ds.
\end{align*}
For the term in the second line we get 
\begin{equs} [e:Gauss17]
\E &\Big| \int_0^t A(\sigma (\alpha^{-1}s,\eps^{-1} z))  
	-  A(\tilde{\sigma} (\alpha^{-1}s,\eps^{-1} z))\, ds \Big|  \\	
&\leq \sup_{\bar{\sigma} \in \{\pm 1 , 0\}}|A(\bar{\sigma})| 
	\int_0^t \P \Big(\sigma(\alpha^{-1}s, \eps^{-1} z) \neq \tilde\sigma(\alpha^{-1}s, \eps^{-1} z) \Big) ds  \\	
&\leq \sup_{\bar{\sigma} \in \{\pm 1 , 0\}}|A(\bar{\sigma})| \int_0^t \Big( \P(T_o> s) +  C \gamma^{1 - 3 \nu}  \Big) ds\\
&\leq \sup_{\bar{\sigma} \in \{\pm 1 , 0\}}|A(\bar{\sigma})| \int_0^t \Big( e^{-\frac{s}{\alpha}} +  C \gamma^{1 - 3 \nu}  \Big) ds\\
&\leq \sup_{\bar{\sigma} \in \{\pm 1 , 0\}}|A(\bar{\sigma})|  \big( \alpha + Ct \gamma^{1 - 3 \nu}  \big) \;.
\end{equs}
Here  $T_o$ is the holding time before the first jump.

For the other term, by Lemma~\ref{lem:sigtildeDecay}, its second moment can be bounded as
\begin{equs}
\E &  \Big( \int_0^t A(\tilde{\sigma} (\alpha^{-1}s,\eps^{-1} z)) ds -\frac{2t}{\beta_c} \Big)^2\\
& \le \int_0^t \int_0^t \E \bar A (\tilde\sigma (\alpha^{-1}s,\eps^{-1}z))\bar A (\tilde\sigma (\alpha^{-1}s',\eps^{-1}z))\,ds\,ds' \\
& \le C\int_0^t \int_0^t e^{-\frac{|s-s'|}{\alpha}} \,ds\,ds' \le C\alpha \;.
\end{equs}
So this term goes to zero as well. Therefore we have shown that the error term in \eqref{e:Epppp} goes to zero and thus the theorem is proved.
\end{proof}

The following result will also be applied several times in the sequel. 
\begin{proposition}\label{prop:averageA}
For every $0\le s \le t\leq T$ and $x \in \Le$, one has
\begin{equs} [e:averageA]
 \int_0^s  & \sum_{z \in \Le} \eg^2 \,  \big( \Pg{t-r} \ae \Kg \big)^2(z-x)
 	 \, \sum_{\bar{\sigma} \in \{ \pm 1,0\} } (\bar{\sigma} - \sigma(\alpha^{-1}r,\eps^{-1} z))^2 
	 \Cgn\big(r,  z, \bar{\sigma}) \, dr \\
& =  \frac{2}{\beta_c} \int_0^s \sum_{z \in \Le} \eg^2 \,  \big( \Pg{t-r} \ae \Kg \big)^2(z-x)\, dr
+ \tilde E_t(s,x)
\end{equs}
where the process $\tilde E$ satisfies the bound 
\begin{equ}[e:Error-stochastic-1]
\E    |\tilde E_t(s,x)|^p  \le C\gamma^{1-3\nu}\log (\gamma^{-1})^{p+1}
\end{equ}
 for every $p\ge 2$ and
 some constant $C=C(T,\nu,\nn)$ depending linearly on $\nn$. Its extension $   \Ex \tilde E_t(s, \cdot)$, which will still be denoted by $\tilde E_t(s,\cdot)$, satisfies
 \begin{equ}[e:Error-stochastic-2]
 \E    \|\Ex \tilde E_t(s, \cdot)\|_{L^p(\T^2)}^p  \le C\gamma^{1-4\nu}\log (\gamma^{-1})^{2p} 
 \end{equ} 
for every $p\ge 2$ and
 some constant $C=C(T,\nu,\nn)$ depending linearly on $\nn$.
\end{proposition}

\begin{proof}
We first show  that the sum over $\bar\sigma$  can be replaced by $A(\sigma(r,\eps^{-1}z))$ (recall the definition of $A$ in \eqref{e:Gauss4})
up to  an error which is controlled deterministically.
Turning to Fourier space, using \eqref{e:Fourier-semi} and 
Parseval's identity and the elementary bound 
$\int_0^s e^{-(s-r)a}dr \le C(\frac{1}{s}+a)^{-1}$ for any $a>0$, we obtain
\begin{equ}  [e:PKsq]
 \int_0^s   \sum_{z \in \Le} \eg^2 \,  \big( \Pg{t-r} \ae \Kg \big)^2(z)\, dr 
 \le C\!\!\!\!\!\!
 \sum_{\om \in \{ -N, \ldots, N \}^2} \frac{|\hKg(\om)|^2}{t^{-1} + 2 \ga^{-b} (1-\hKg(\om))}
\end{equ}
where $b=2$ in the first regime and $b=4$  in the second regime. We then use the estimates \eqref{e:K2.4} and the first estimate in \eqref{e:K1} to bound the sum over $|\om| \leq C\ga^{-1}$ (resp. $C\ga^{-2}$) and the estimate \eqref{e:K3B} to bound the sum over $|\om| \geq C\ga^{-1}$ (resp. $C\ga^{-2}$)  in the first (resp. second) regime, which permits to conclude that the right hand side of \eqref{e:PKsq} is bounded by $C \log(\ga^{-1})$.
Therefore, invoking \eqref{e:Gauss3}, the   left hand side of \eqref{e:averageA} is equal to
\begin{equ} 
 \int_0^s   \sum_{z \in \Le} \eg^2 \,  \big( \Pg{t-r} \ae \Kg \big)^2(z-x)
 	\,A(\sigma(\alpha^{-1} r,\eps^{-1}z))\, dr 
\end{equ}
plus an error which is deterministically bounded by $C\gamma^{1-3\nu}\log \gamma^{-1}$.

We proceed as in the proof of Theorem~\ref{t:converg-lin}, again making use of the process $\tilde{\sigma}$ constructed at the beginning of this section. Arguing as in \eqref{e:Gauss17}
we can replace $A(\sigma(\alpha^{-1} r,\eps^{-1}z))$ in the above integral by $A(\tilde\sigma(\alpha^{-1} r,\eps^{-1}z))$ with an error
satisfying the following {\it first} moment bound
\begin{equs} [e:1stmomA-A]
\E   \int_0^s   \sum_{z \in \Le} \eg^2 \,  \big( \Pg{t-r} \ae \Kg \big)^2(z-x)
 	\, \Big| A(\sigma(\alpha^{-1} r,\eps^{-1}z))  - A(\tilde\sigma(\alpha^{-1} r,\eps^{-1}z)) \Big|   \, dr \\ 
 \leq \sup_{\bar{\sigma} \in \{\pm 1 , 0\}}|A(\bar{\sigma})|
 \int_0^s   \sum_{z \in \Le} \eg^2 \,  \big( \Pg{t-r} \ae \Kg \big)^2(z-x)
 	\,  \big( e^{-\frac{r}{\alpha}} +  C \gamma^{1 - 3 \nu}  \big)  \, dr \;.
\end{equs}
We claim that by a similar argument to the one leading to  \eqref{e:PKsq}, the right hand side of \eqref{e:1stmomA-A} can be bounded by $C\gamma^{1-3\nu}\log \gamma^{-1}$.
 Indeed, for  the term involving $C \gamma^{1 - 3 \nu} $ this is immediately clear from the above $\log(\ga^{-1})$ bound on  \eqref{e:PKsq}. For  the term with $e^{-\frac{r}{\alpha}}$  
we divide the $r$-integral into an integral over $r\in [\ga,s]$ and an integral over $r\in [0,\ga]$.
For the integral over $r\in [\ga,s]$,
we simply bound $e^{-\frac{r}{\alpha}}\le C\ga$ (recall that $\alpha \approx \gamma^2$ in the first and $\alpha \approx \gamma^4$ in the second scaling regime), and the integration of the other factors is bounded by $C\log(\ga^{-1})$ as above.
For the integral over $r\in [0,\ga]$, we bound $e^{-\frac{r}{\alpha}}\le 1$, 
and then since after applying Parseval's identity the only $r$-dependent factor
inside the $r$-integral is
$e^{-2(t-r) \gamma^{-b} (1-\hKg(\om) )}$ and as this function is monotonically increasing in $r$,  we have
\[
\int_0^\ga  e^{-(t-r) \gamma^{-b} (1-\hKg(\om) ) } dr
\le \frac{\ga}{s} \int_0^s  e^{-(t-r) \gamma^{-b} (1-\hKg(\om) ) }dr \;;
\]
applying the above $\log(\ga^{-1})$ bound again we conclude that as claimed the right hand side of \eqref{e:1stmomA-A} is bounded by $C\gamma^{1-3\nu}\log \gamma^{-1}$.

Finally using the deterministic bound
\begin{equs}
\int_0^s   \sum_{z \in \Le}&   \eg^2 \,  \big( \Pg{t-r} \ae \Kg \big)^2(z-x)
 	\, \big| A(\sigma(\alpha^{-1} r,\eps^{-1}z))  - A(\tilde\sigma( \alpha^{-1} r,\eps^{-1}z)) \big|   \, dr \\
	& \leq C \log \gamma^{-1},
\end{equs}
the above bound on the {\it first} moment can be upgraded to a bound on {\it all} stochastic moments. We get for any $p \geq 1$ that 
\begin{align}
\notag
&\E   \Big(\int_0^s   \sum_{z \in \Le} \eg^2 \,  \big( \Pg{t-r} \ae \Kg \big)^2(z-x)
 	\, \Big| A(\sigma( \alpha^{-1} r,\eps^{-1}z))  - A(\tilde\sigma(\alpha^{-1} r,\eps^{-1}z)) \Big|   \, dr\Big)^p \\
&	\le  C \gamma^{1- 3 \nu} ( \log \gamma^{-1})^p.\label{Gaussian88}
\end{align}

To prove \eqref{e:averageA}
it remains to control moments of the  error term
\begin{equ}
 \int_0^s \sum_{z \in \Le} \eg^2 \,  \big( \Pg{t-r} \ae \Kg \big)^2(z-x)
 \Big(A(\tilde\sigma(\alpha^{-1} r,\eps^{-1}z))-  \frac{2}{\beta_c} \Big)\, dr \;.
\end{equ}
As before we use the centered random field $\bar A(\tilde\sigma(\alpha^{-1}r,\eps^{-1}z)) =A(\tilde\sigma(\alpha^{-1} r,\eps^{-1}z))-  \frac{2}{\beta_c}  $ and write
\begin{align*}
&\E  \Big(  \int_0^s \sum_{z \in \Le} \eg^2 \,  \big( \Pg{t-r} \ae \Kg \big)^2(z-x)
 \bar A(\tilde\sigma(\alpha^{-1} r,\eps^{-1}z)) \, dr \Big)^2\\
 &=    \int_0^s \int_0^s \sum_{z \in \Le}    \sum_{z' \in \Le} \eg^4    \,  \big( \Pg{t-r} \ae \Kg \big)^2(z-x) \big( \Pg{t-r'} \ae \Kg \big)^2(z'-x)\\
& \qquad \qquad \times \E   \bar A(\tilde\sigma(\alpha^{-1} r,\eps^{-1}z))\,  \bar A(\tilde\sigma(\alpha^{-1} r',\eps^{-1}z'))  \,dr  dr'.
\end{align*}
Applying Lemma~\ref{lem:sigtildeDecay}, this turns into
\begin{align*}
&\E  \Big(  \int_0^s \sum_{z \in \Le} \eg^2 \,  \big( \Pg{t-r} \ae \Kg \big)^2(z-x)
 \bar A(\tilde\sigma(\alpha^{-1} r,\eps^{-1}z))\, dr \Big)^2\\
 &\leq C \eg^2  \int_0^s \int_0^s \sum_{z \in \Le}    \eg^2    \,  \big( \Pg{t-r} \ae \Kg \big)^2(z-x) \big( \Pg{t-r'} \ae \Kg \big)^2(z-x) \,e^{-\frac{|r-r'|}{\ag}}\,  dr  dr'\\
 &\leq C \eg^2 
 	\sup_{r'\in[0,s]}\|  \Pg{t-r'} \ae \Kg  \|_{L^\infty(\Le)}^2 \\
		& \qquad\qquad \times \int_0^s  \sum_{z \in \Le}    \eg^2    \,  \big( \Pg{t-r} \ae \Kg \big)^2(z-x)  \Big(\int_0^s \,e^{-\frac{|r-r'|}{\ag}}\,  dr' \Big) dr\\
 &\leq C \eg^2 \left( \gamma^{-b} \log(\ga^{-1})\right)^2 \alpha \int_0^s  \sum_{z \in \Le}    \eg^2    \,  \big( \Pg{t-r} \ae \Kg \big)^2(z-x)   dr\\
 & \leq C  \eg^2 \gamma^{-2b} (\log(\ga^{-1}))^3 \alpha,
\end{align*}
where in the third inequality we have used \eqref{e:P0} and $b = 2 $ in the first regime and $b= 4$ in the second regime. In both the first regime \eqref{e:Phi4-scaling} and  the second regime  \eqref{e:Phi6-scaling} this expression is bounded by 
$\leq C\gamma^2 (\log(\ga^{-1}))^3$.
As before we can upgrade this stochastic $L^2$ to a stochastic $L^p$ bound 
by using a deterministic bound 
\begin{equ}
\int_0^s \sum_{z \in \Le} \eg^2 \,  \big( \Pg{t-r} \ae \Kg \big)^2(z-x)
 \bar A(\tilde\sigma(\alpha^{-1} r,\eps^{-1}z)) \, dr \leq C \log \gamma^{-1}.
\end{equ}
Therefore in both scaling regimes
 \eqref{e:Error-stochastic-1} follows.

To obtain the second bound (Eq.~ \eqref{e:Error-stochastic-2})
we sum \eqref{e:Error-stochastic-1} over $x \in \Le$ to obtain 
\begin{equ}
\E    \|\tilde E_t(s,\cdot) \|_{L^p(\Le)}^p  = \sum_{x \in \Le} \eps^2 \E    |\tilde E_t(s,x) |^p \leq   C\gamma^{1-3\nu}\log (\gamma^{-1})^p.
\end{equ}
To replace the $L^p$ norm over $\Le$ by the $L^p$ norm over the continuous torus and $\tilde E$ by its extension write using Jensen's inequality
\begin{equs}[e:AAA1]
&\int_{\T^2} |\Ex \tilde E_t(s,z)|^p dz   \\
&= \int_{\T^2} \Big| \sum_{x \in \Le} \eps^2 \tilde E_t(s,x) \mathsf{Ker}(x-z) \Big|^p dz \\
 &\leq \int_{\T^2} \Big( \sum_{x \in \Le} \eps^2 |\tilde E_t(s,x) |^p |\mathsf{Ker}(x-z)| \Big) \Big( \sum_{x \in \Le} \eps^2 |\mathsf{Ker}(x-z)| \Big)^{p-1} dz
\end{equs}
where (as discussed in  \cite[Lemma~A.6]{MourratWeber}) the extension kernel is given by 
\[
\mathsf{Ker}(x-z) = \prod_{j=1}^2 \frac{\sin \big(\frac{\pi}{2}(2N+1) (x_j - z_j) \big)}{\sin \big(\frac{\pi}{2}(x_j - z_j) \big)}
\]
so that we have  that  $\sum_{x \in \Le} \eps^2 |\mathsf{Ker}(x-z)| \le C\log \gamma^{-1}$ uniformly in $z$. Plugging this estimate into \eqref{e:AAA1}
yields
\begin{align*}
&\int_{\T^2} |\Ex \tilde E_t(s,z)|^p dz   \\
& \leq C(  \log \gamma^{-1})^{p-1}  \Big( \sum_{x \in \Le} \eps^2 |\tilde E_t(s,x) |^p  \int_{\T^2} |\mathsf{Ker}(x-z)|\,dz \Big) \\
& \leq C(  \log \gamma^{-1})^{p} \| \tilde E_t(s, \cdot)\|_{L^p(\Le)}^p
\end{align*}
so  \eqref{e:Error-stochastic-2} follows as well.
%
\end{proof}

\section{Wick powers and proof of the main theorem}
The aim of this section is to prove Theorem~\ref{thm:Main}.
Since we will apply a discrete version of Da Prato-Debussche argument (\cite{dPD}) as in \cite{MourratWeber}, an important step is to
prove the convergence of the approximate Wick powers $\ZG{n}$ to the  Wick powers. 
Fortunately, the work \cite{MourratWeber} treated the Wick powers with general $n$, though only $n\le 3$ was needed therein; here we only need some minor modifications 
to their construction of Wick powers.

%

We start by recalling the definitions of the approximate Wick powers $\ZG{n}$.
Recall that $\Zg$ is defined in \eqref{e:Z-gamma}.
It will be convenient to work with the following family of approximations to $\Zg(t,x)$. For $s \leq t$, we introduce
\begin{equ}
\Rg(s,x) \eqdef \int_{r=0}^s  \Pg{t-r} \, d\Mg (r,x) \;,
\end{equ} 
and  extend $\Rg(s, \cdot) $ and $\Zg( t , \cdot) $ to  functions on all of $\T^2$ by trigonometric polynomials of degree $\leq N$ as \eqref{e:Extension}. Note that for any $t$ and any $x \in \T^2$, the process $\Rg(\cdot,x)$ is a martingale and $\Rg(t,\cdot) = \Zg(t,\cdot)$. 

The iterated integrals are then defined recursively as follows.
For a fixed $t\geq 0$ and $x \in \T^2$, we set $\RG{1}(s,x) = \Rg(s,x)$. For $n \geq 2, \; t \geq 0$  and $x \in \Le$, we set
\begin{equation}\label{e:ZnA}
\RG{n}(s,x) =  n \int_{r=0}^s \RG{n-1}(r^-,x) \; d \Rg(r,x)\;.
\end{equation}
We use the notation $\RG{n-1}(r^-,x)$ to denote the left limit of $\RG{n-1}(\cdot,x)$ at $r$. This definition ensures that $(\RG{n}(s,x))_{0 \le s \le t}$ is a martingale. 
The extension of $\RG{n}(s, \cdot)$ to the entire $\T^2$ is also 
defined  recursively, through its Fourier series
\begin{equation}\label{e:ZnB}
\hRG{n}(s, \om) \eqdef n\int_{r=0}^s \frac{1}{4}\sum_{\tilde{\om} \in \Z^2} \hRG{n-1}(r^-, \om- \tilde{\om}) \; d\hRg(r, \tilde{\om})\;,
\end{equation} 
and set $\RG{n}(s,x) \eqdef \frac14 \sum_{\om \in \Z^2}\hRG{n}(s,\om)e^{i \pi \om \cdot x}$. This definition coincides with \eqref{e:ZnA} on $\Le$, and for every $n \geq 2$ the function $\RG{n}(s, \cdot) \colon \T^2 \to \R$ is a trigonometric polynomial of degree $\leq nN$. For any $n \geq 2$ and for $t \geq 0$, $x \in \T^2$ we define
\begin{equation}\label{e:DefZn}
\ZG{n}(t,x) \eqdef \RG{n}(t,x) \;.
\end{equation}

Finally let $\RGn{n}$ and $\ZGn{n}$ be iterated stochastic integrals  defined just as $\RG{n}$ and $\ZG{n}$  but with $\Mg$ replaced by $\Mgn$. Recall that $\nn$ is the parameter fixed in \eqref{e:deftaug}.

By the definition of $\Rg(s,x)$ and the quadratic variation of $M_\gamma$, one has
\begin{equs} [e:QuadrVarR]
\langle \Rg(\cdot, x) \rangle_s =  \co^2 \int_0^s &\sum_{z \in \Le} \eps^2  \big(\Pg{t-r} \ae \Kg \big)^2 (x-z)  \\
&\times\sum_{\bar{\sigma} \in \{ \pm 1,0\} } (\bar{\sigma} - \sigma(r,\eps^{-1} z))^2 
	 \Cgn\big(r,  z, \bar{\sigma}) \, dr \;.
\end{equs}

There exists a constant $ \ga_0>0$ (arising when we apply the kernel bounds in Section~\ref{sec:kernels}) such that the following results hold.

\begin{proposition}\label{prop:RGBoundOne}
 For every $n \in \N$,  $p\geq 1$,  $\al>0$,  $T>0$,  $0 \leq \la \leq \frac12$ and $0 < \ka \leq 1$,  there exists a constant $C= C(n,p,\al,T,  \la,\ka)$ such that for every $0 \leq s \leq t \leq T$ and $0< \ga < \ga_0$, one has
\begin{align}
 \E \sup_{0 \leq r \leq t}\| \RG{n}(r, \cdot) \|^p_{\Cc^{-\al- 2\la}} &\leq C \,t^{\la \, p}    +C \ga^{p(1 - \ka)}   \;, \label{e:RgRegularity1}\\
\E \sup_{0 \leq r \leq t} \| \RG{n}(r, \cdot) -R_{\ga,s}^{:n :}(r \wedge s, \cdot) \|^p_{\Cc^{-\al-  2\la}} &\leq C \, |t-s|^{\la \, p}  +C \ga^{p(1 - \ka)} \; ,  \label{e:RgRegularity2} \\
\E \sup_{0 \leq r \leq t} \| \RG{n}(r, \cdot) - \RG{n}(r \wedge s,\cdot) \|^p_{\Cc^{-\al-  2\la}} &\leq C  \; |t-s|^{\la \, p}   +C \ga^{p(1 - \ka)} \; .  \label{e:RgRegularity3} 
 \end{align}
The same bounds hold for $\RGn{n}$.
\end{proposition}

\begin{proposition} \label{prop:boundQ}
For $x \in \Le$, let
\begin{equation}
\label{e:def:Qg}
\Qg(s,x) = [\Rg(\cdot,x)]_s - \langle \Rg(\cdot,x) \rangle_s\;.
\end{equation}
For any $t \ge 0$, $\ka > 0$ and $1 \le p < +\infty$, there exists $C = C(t,\ka,p)$ 
 such that for $0 < \ga < \ga_0$,
$$
 \E\sup_{x \in \Le} \sup_{0 \le s \le t} |\Qg(s,x)|^p  \le C \ga^{p(1-\ka)}.
$$
The same bound holds for $Q_{\ga,t,\nn}$, that is,
the same process as $Q_{\ga,t}$ but defined via $\Mgn$ instead of
 $\Mg$.
\end{proposition}

One important result is that these iterated integrals are almost Hermite polynomials with renormalization constant chosen as $[\Rg(\cdot,x)]_s$.

\begin{proposition} 
\label{p:discrete-Wick}
Define
\begin{equation}
\label{e:def:EGn}
\EG{n}(s,x) \eqdef H_n(\Rg(s,x), [\Rg(\cdot,x)]_s) - \RG{n}(s,x)\;,
\end{equation}
for any $x \in \T^2$. Here, we view $[\Rg(\cdot,x)]_s$ as defined on all of $\T^2$, by extending it as a trigonometric polynomial of degree $\leq N$. 
Then  for any $n \in \N$, $\ka > 0$, $t > 0$ and $1 \le p < \infty$, there exists $C =C(n,p,t,\ka)> 0$ such that for every sufficiently small $\ga >0$,
\begin{equation}
\label{e:discrete-Wick}
 \E \sup_{x \in \T^2} \sup_{0 \le s \le t} | \EG{n}(s,x) |^p  \le C \ga^{p(1-\ka)} \notag.
\end{equation}
The same bound holds for $E_{\ga,t,\nn}^{:n:}$ - the same process as $\EG{n}$ but defined via $\Mgn$ instead of
 $\Mg$.
\end{proposition}

\begin{proof}[of Prop.~\ref{prop:RGBoundOne} - \ref{p:discrete-Wick}]
For the case of the Kac Ising model, these results are Prop~4.2, Lemma~5.1 and Prop~5.3
in \cite{MourratWeber}.
Several  modifications of these proofs are necessary for the case of our Blume-Capel model.  

The first necessary modification is due to the difference in the scalings \eqref{e:Phi4-scaling} and \eqref{e:Phi6-scaling}. 
This difference  comes into play via the estimates on the kernels $K_\ga$ and $P^\ga_t$ 
used throughout the proofs. We list all these kernel estimates in Section~\ref{sec:kernels}. These estimates with modifications in the second regime 
lead to the desired bounds {\it mutatis mutandis}.


Another necessary modification of the proof for the case of our Blume-Capel model
is due to the fact that the martingale we use to build $\ZG{n}$ is different.
For Proposition~\ref{prop:RGBoundOne}, the only place where the martingale enters into play is \cite[Lemma~4.1]{MourratWeber}, which is a consequence of Burkholder-Davis-Gundy inequality.
The proof of that lemma only used two facts that depend on the martingale. 
First, a jump of the spin at $\eps^{-1}z$ causes a jump of size $2\delta^{-1}\eps^2 K_\ga(y-z)$ for $M_\ga(y)$, and in our case this becomes an upper bound of the jump size since a spin could jump by $1$ or $2$.
 Second, in the quadratic variation of $\Mg$
which was given by 
\begin{equ}
\frac{d}{dt} \langle \Mg(\cdot, x ), \Mg (\cdot, y) \rangle_t  
=4 c_{\gamma,2}^2 
\sum_{z \in \Le}\eg^2  \Kg( x -z) \, \Kg(y-z) \,  \Cg\big(t, z\big)  \;, 
\end{equ} 
and
$C_\gamma $ is a  rate function therein which is bounded between $0$ and $1$.
For our case, in the quadratic variation given in \eqref{Gauss2}, one also has
\begin{equ} [e:rate-bnd5]
0\le
\sum_{\bar{\sigma} \in \{ \pm 1,0\} } (\bar{\sigma} - \sigma(\alpha^{-1} s, \eps^{-1}z))^2 \Cg\big(s,  z, \bar{\sigma})
\le 5 \;.
\end{equ}
Since the desired bound in \cite[Lemma~4.1]{MourratWeber} allows a proportionality constant, nothing else needs to be proved.

For  Proposition~\ref{prop:boundQ}, by Burkholder-Davis-Gundy inequality, one needs to bound the quadratic variation $\langle \Qg(\cdot,x) \rangle_t$,
which can be again explicitly expressed as in the case for $\Rg(\cdot, x)$ in \eqref{e:QuadrVarR}; using the bound \eqref{e:rate-bnd5} one eventually obtains
\begin{equ}
\langle \Qg(\cdot,x) \rangle_t
 \le  \frac{C\eps^6}{\ag \dg^4 } \int_0^t \sum_{z \in \Le} \eps^2 (\Pg{t-s} \ae  \Kg)^4(z) \, ds\;.
\end{equ}
Using the bound $\|\Pg{t-s} \ae  \Kg\|_{L^\infty(\Le)} \le C\frac{\ga^2}{\eps^2}$
and $(\eps^2 \ga^4 / \ag \dg^{4}) \le 2\ga^2$ which turn out to hold in {\it both} regimes,
the proof of \cite[Lemma~5.1]{MourratWeber} again goes through.

Proposition~\ref{p:discrete-Wick} is then a consequence of the first two propositions by the proof in \cite{MourratWeber},
and therefore nothing needs to be re-proved.
\end{proof}

One then has the following tightness and convergence results.

\begin{proposition}\label{prop:tight}
For every $\nn \in \N$ and  $\al >0$, the family $\{ Z_{\ga,\nn}^{:n:}, \ga \in (0,\frac13) \}$ is tight on $\Dd(\R_+,\Ca) $. Any weak limit is supported on $\Cc(\R_+,\Ca) $. Furthermore, for any $p \geq 1$ and $T>0$, we have
\begin{equation}\label{e:FinalAPriori}
\sup_{\gamma \in (0,\frac13)}
	 \E  \sup_{0 \leq t \leq T} 
	  \big\| Z_{\ga,\nn}^{:n:}(t, \cdot ) \big\|_{\Ca}^p   < \infty\;.
\end{equation}
\end{proposition}

\begin{proof}
Once Proposition~\ref{prop:RGBoundOne} (in particular the bounds \eqref{e:RgRegularity1}
and \eqref{e:RgRegularity2}) is shown, this tightness result follows in exactly the same way as
 \cite[Proposition~5.4]{MourratWeber}.
\end{proof}

Recall that we have defined $Z^{: m:}$ below \eqref{e:def-Wick-eps}.
\begin{proposition}
\label{t:converg-lin-Wickbis}
For every $\nn \in \N$ and $n \in \N$, the processes $(\ZGn{1},$ $\ldots,  \ZGn{n})$  defined above converge (jointly) in law to $(\ZZ{1}, \ldots, \ZZ{n})$ with respect to the topology of  $\Dd(\R_+, \Ca)^n$.
\end{proposition}

\begin{proof}
Since by Proposition~\ref{prop:tight} for every $n$, the family of vectors $(\ZGn{1}, \ldots, \ZGn{n})$, $\ga \in (0,\frac13)$ is tight with respect to the topology of $\Dd(\R_+,\Ca)^n$,
we only need to show convergence of the finite dimensional distributions. 
We follow the diagonal argument
as in \cite[Theorem~6.2]{MourratWeber}.
Define
$$
R_t(s,x) \eqdef \sqrt{2/\beta_c} \int_{r=0}^s  P_{t-r} \,  dW (r,x) \;,
$$
where $\beta_c$ is a critical value of $\beta$ as above. 
The process $s \mapsto R_t(s,x)$ for $s < t$ is a {\it continuous} martingale.
For $n>1$ define 
\begin{equation}\label{e:Wick0}
\RR{n}(s,x) \eqdef n \int_{r=0}^s \RR{n-1}(r,x) \; d R_t(r,x) =  H_n\Ll(R_t(s,x), \langle R_t(\cdot,x) \rangle_s\Rr)  \;.
\end{equation}
For $s<t$ $\RR{n}(s,x)$ is a regular approximations of  the limiting objects $\ZZ{n}(t, \cdot) $; indeed, as discussed in \cite[(3.10)]{MourratWeber},
 for all $\al>0, \; 0 \leq  \la \leq 1, \; p \geq2$ and $T >0$,  there exists $C= C(\al,\la, p,T)$ such that 
\begin{equation}
\label{e:regR0n}
\E\| \ZZ{n}(t, \cdot) - \RR{n}(s,\cdot) \|^p_{\Cc^{-\al-\la}} \le C |t-s|^{\frac{\lambda p}{2}}
\end{equation}
for all $0 \leq s \leq t \leq T$. Write 
\begin{equs}
\bZg = (\ZGn{1}, \ldots, \ZGn{n})\;, &\qquad \bZ = (\ZZ{1}, \ldots, \ZZ{n})\;, \\
\bRg{t} = (\RGn{1}, \ldots, \RGn{n}) \;, &\qquad \bR = (\RR{1}, \ldots, \RR{n}) \;.
\end{equs}
Fix $K \in \N$ and $t_1 < t_2 < \ldots < t_K$. Let $F\colon (\Ca)^{n \times K} \to \R$ be bounded and uniformly continuous.  
For $s_1 < t_1, \,  \ldots, s_K < t_K$, 
\begin{equs}[e:CiL1]
\big| &\E\, F\big(\bZg(t_1), \ldots, \bZg(t_K)\big)  - \E \,F\big(\bZ(t_1), \ldots, \bZ(t_K)\big) \big| \notag \\
&\leq  \E\,\big|  F\big(\bZg(t_1), \ldots, \bZg(t_K)\big)  -\,F\big(\bRg{t_1}(s_1), \ldots, \bRg{t_K}(s_K)\big) \big| \notag \\
& \qquad + \big| \E\, F\big(\bRg{t_1}(s_1), \ldots, \bRg{t_K}(s_K)\big)  - \E \,F\big(\mathbf{R}_{t_1}(s_1), \ldots, \mathbf{R}_{t_K}(s_K)\big) \big|\notag \\
& \qquad +  \E\,\big|  F\big(\mathbf{R}_{t_1}(s_1), \ldots, \mathbf{R}_{t_K}(s_K)\big)  -\,F\big(\bZ(t_1), \ldots, \bZ(t_K)\big) \big| \;. 
\end{equs}
The estimates \eqref{e:regR0n} and \eqref{e:RgRegularity3} yield moment bounds of arbitrary order of $\| \bZg(t_i) - \bRg{t_i}(s_i) \big\|_{(\Ca)^n}$ uniformly in~$\ga$. We can thus make the first and the third terms on the right-hand side of \eqref{e:CiL1} small uniformly in $\ga$ by choosing $|t_i - s_i|$ small enough. 

Some extra care has to be taken in the case of our model for the second term on the right-hand side of \eqref{e:CiL1}.
By Proposition~\ref{p:discrete-Wick}, it suffices to show that 
$$
H_\ell(R_{\ga,t_i,\nn}(s_i,x), [R_{\ga,t_i,\nn}(\cdot,x)]_{s_i}) \qquad \ell = 1, \ldots, n, \quad i = 1, \ldots, K 
$$
converges in law to $(\mathbf{R}_{t_1}(s_1), \ldots, \mathbf{R}_{t_K}(s_K))$ in $(\Ca)^K$. By \eqref{e:Wick0} and Prop~\ref{prop:boundQ}, it suffices to show the two convergences in law
\begin{equs}
\big( R_{\ga,t_1,\nn}(s_1) , \ldots, R_{\ga,t_K,\nn}(s_K) \big) & \xrightarrow[\ga \to 0]{} \big(R_{t_1}(s_1), \ldots, R_{t_K}(s_K) \big)\; , \\
\big( \langle R_{\ga,t_1,\nn}(\cdot,\cdot)\rangle_{s_1}, \ldots, \langle R_{\ga,t_K,\nn}(\cdot,\cdot)\rangle_{s_K} \big)  &\xrightarrow[\ga \to 0]{} \big( \langle R_{t_1}(\cdot,\cdot)\rangle_{s_1}, \ldots, \langle R_{t_K}(\cdot,\cdot)\rangle_{s_K} \big) \;, 
\end{equs}
for a suitable topology, e.g. $(L^\infty)^K$ in the first convergence and $(L^p)^K$ for $p$ large enough for the second convergence. 
For the first convergence,
note that  $R_{\ga,t_i,\nn}(s_i) = P_{t_i-s_i}^\ga Z_{\ga,\nn}(s_i)$. \cite[Corollary~8.7]{MourratWeber} then gives an error control if $P_{t_i-s_i}^\ga$ is replaced by the continuous heat kernel $P_{t_i-s_i}$. 
So the first convergence  follows from Theorem~\ref{t:converg-lin} (convergence of $Z_\ga(t)$), continuity of the mapping  $P_{t_i-s_i}$ and  the continuous mapping theorem.

Regarding the second convergence,
recall the explicit expression \eqref{e:QuadrVarR} for the quadratic variation
$\langle R_{\ga,t_i,\nn}(\cdot,x)\rangle_{s_i} $. The constant $\co^2$ is deterministically close to $1$ by \eqref{e:const-c-gamma},
and therefore Proposition~\ref{prop:averageA} shows that
the quadratic variation $\langle R_{\ga,t_i,\nn}(\cdot,x)\rangle_{s_i}$ is given by
$$
\frac{2}{\beta_c} \int_0^{s_i}    \sum_{z \in \Le} \eg^2 \,  \big( \Pg{t_i-r} \ae \Kg \big)^2(z-x) \, dr
$$
up to an error $\tilde E_t(s)$ which satisfies $\E \| \tilde E_t(s)\|_{L^p(\T^2)}^p  \to 0 $. 
This expression in turn converges to the limiting object $\langle R_{t_i}(\cdot,\cdot)\rangle_{s_i}$ by the calculation as in \cite[(6.14)]{MourratWeber}.
\end{proof}

We now summarize the results obtained above and prove our main result, Theorem~\ref{thm:Main}. 


To show the convergence 
of discrete evolution \eqref{e:mildBEG6}  to the solution of 
\begin{equs}  [e:mildPhi4]
X(t,\cdot) =P_t X^0 &+ \int_0^t P_{t-s}   \star  
	 \Big( \mathfrak a_1 X(s,\cdot) 
	 -\frac{a_c(4a_c-1) \beta_c^3 }{3(2a_c+1)^2}  X^{:3:} (s,\cdot)
	 \Big) \, ds \\ & + Z(t, \cdot) 
  \qquad \mbox{on } \T^2
\end{equs}
in the first regime and 
\begin{equs}  [e:mildPhi6]
X(t,\cdot) =&P_t X^0 + \int_0^t P_{t-s}   \star  
	 \Big( -\frac{20}{9}X^{: 5:}  (s,\cdot) + \mathfrak{a}_3  X^{: 3:}  (s,\cdot) \\
+ & \mathfrak{a}_1 X(s,\cdot) \Big) \, ds  + Z(t, \cdot) 
  \qquad \mbox{on } \T^2
\end{equs}
in the second regime,
we need to control the following  error terms.

(1) The error $E_\gamma$ in \eqref{e:mildBEG6} arising from the Taylor expansion in Section~\ref{sec:model-formal}.

(2) In the second regime the discrepancies caused by $C_{\beta,\theta} \neq -\frac{20}{9}$,
  the coefficient in front of $X_\gamma^3 -3\CGG X_\gamma$
 in \eqref{e:reorg-Hermite}
is not exactly $\mathfrak a_3 $, and  the coefficient in front of 
$X_\gamma$ in \eqref{e:reorg-Hermite} is not exactly $\mathfrak a_1 $;
similarly in the first regime there are also such discrepancies of coefficients comparing with \eqref{e:reorg-Hermite-1}.

(3) The operator $\Ex$ which extends a function on $\Lambda_\eps$ to a function on $\T^2$
defined in  \eqref{e:Extension} does not commute with powers. As in \cite{MourratWeber} this is dealt with by decomposing the field $\Xg$ into a ``high" and a ``low" frequency part
\begin{equ} [e:defXlh]
 \Xg^{\mathrm{low}} \eqdef \sum_{2^k < \frac{N}{20}} \delta_k \Xg \;,
 \qquad
 \Xg^{\mathrm{high}} \eqdef \sum_{2^k \geq \frac{N}{20}} \delta_k \Xg \;,
 \end{equ}
where we refer to \cite[(A.7)]{MourratWeber} for the precise definition of the operator $\dk$
(we recall that $N\approx \ga^{-2}$ in the first regime and $N\approx \ga^{-3}$ in the second regime).  For $\Xg^{\mathrm{low}}$ the operator $\Ex$ does commute with the powers appearing below and we need to control the error caused by the high frequencies. 

(4) Recall that in the discussion on the limiting SPDE, the actual renormalization constant used to define the Wick powers $Z_\eps^{: n :}$ in \eqref{e:def-Wick-eps}
is a time-dependent constant $\mathfrak c_\eps(t)$, and the time-dependent coefficients $\mathfrak a_k(t)$ is introduced in place of the time-independent ones $\mathfrak a_k$ in order to take care of the difference between $\mathfrak c_\eps(t)$ and $\mathfrak c_\eps$, i.e. to
guarantee that \eqref{e:a-and-at} holds. For the discrete model, we have $\mathfrak c_\gamma \neq \mathfrak c_\eps$, and  we will introduce the approximate time-dependent renormalization constant 
\begin{equ} [e:def-cga-s]
\tCG(s,x) \eqdef [ R_{\ga,s}(\cdot,x)]_{ s}
\end{equ}
 (and extend this to all $x \in \T^2$ as a trigonometric polynomial). So we need to control the error caused by the fact that  Eq.~\eqref{e:a-and-at} does not exactly hold
 anymore if the subscript $\eps$ in \eqref{e:a-and-at} is replaced by $\gamma$.

(5) The error from $P_t^\gamma X_\gamma^0 \neq P_t X^0$.

(6) The processes $\Zgn^{: n:}$ are defined via iterated integrals, which are not exactly the same as Hermite polynomials with constant $\tCG(s,x)$ (see Prop.~\ref{p:discrete-Wick}).

(7) $\Delta \neq \widetilde{\Delta}_\gamma$.

In the following Lemma we control the errors from (1)-(4). We will frequently use the fact that
an $L^\infty (\Lambda_\eps)$ bound on $X_\gamma$ can be extended to an $L^\infty (\T^2)$ bound by loosing an arbitrarily small power of $\gamma$ (\cite[Lemma~B.6]{MourratWeber}), 
and the fact that the $L^\infty$ norm can be bounded by the $\mathcal C^{-\nu}$ norm of $X_\gamma$ multiplied by a factor $\ga^{-b\nu}$ (\cite[Lemma~B.3]{MourratWeber})
if $\hat X_\gamma$ has vanishing frequency larger than $\ga^{-b}$ ($b=2,3$ depending on the regime).

Before stating the lemma, we recall that the constant $\Ce$ is defined in \eqref{e:norm-constant}, the constant $\Ce(t)$  is defined in  \eqref{e:coft},
the constant $\CGG$ is defined in \eqref{e:valueCGG},
the constant $\tCG(t,\cdot)$ is defined in \eqref{e:def-cga-s},
the constant $\mathfrak a_1$ (resp. $\mathfrak a_1$  and $\mathfrak a_3$)  are introduced 
in \eqref{e:Phi4-beta-a} (resp. \eqref{e:tune-tri}) in the first (resp. second) regime. The constants $\mathfrak a_k^{(\eps)}(t)$ are defined in \eqref{e:a-and-at}, and here we will use the $\eps\to 0$ limits of them:
in the second regime, by
 \eqref{e:a3t-a3} and  \eqref{e:a1t-a1} with $\mathfrak a_5$ substituted by $-\frac{9}{20}$ we define $\mathfrak a_1(s)$, $\mathfrak a_3(s)$ as $\eps\to 0$ limits of $\mathfrak a_1^{(\eps)}(s)$, $\mathfrak a_3^{(\eps)}(s)$, namely
 \begin{equ}  [e:a3a3a1a1]
\mathfrak a_3 (s) - \mathfrak a_3 = -\frac92 \bar{\mathfrak c} (s)  \;,\qquad
\mathfrak a_1 (s) - \mathfrak a_1  = 3\mathfrak a_3 \bar{\mathfrak  c}(s)  
	-\frac{27}{4} \bar{\mathfrak c}(s)^2 \;,
\end{equ}
 where $\bar{\mathfrak  c} (s) \eqdef  \lim_{\eps \to 0} (\Ce(s) -\Ce) $  (see \eqref{e:ce-cet} for existence of this limit).
In the first regime we simply define
$\mathfrak a_1 (t)= 
3\mathfrak a_3 \bar{\mathfrak  c} (s)+\mathfrak a_1
=-\frac{a_c(4a_c-1) \beta_c^3 }{(2a_c+1)^2}\bar{\mathfrak  c} (s)+\mathfrak a_1$.

\begin{lemma} \label{lem:Error1}
For every $t \geq 0$, we have on $\T^2$  (we drop the space variables for readability)
\begin{equs} [e:evolution3XX]
\Xg(t ) =\Pg{t} \Xng + \int_0^t & \Pg{t-s}  \Kg \star   
	\Big( -\frac{a_c(4a_c-1) \beta_c^3 }{3(2a_c+1)^2} 
	  \Big( \Xg^3 (s) -3\tCG(s )\Xg(s)\Big)
\\
& +  \mathfrak a_1(s)  \Xg(s )+ \ERR{1}(s ) 
 \Big) \, ds  + \Zg(t ) \;.   
\end{equs}
in the first scaling regime and 
\begin{equs} [e:evolution3CC]
\Xg(t ) =& \Pg{t} \Xng + \int_0^t  \Pg{t-s}  \Kg \star   
	\Big( - \frac{9}{20}  \Big( \Xg^5 (s )-10\tCG(s )\Xg^3(s)+15\tCG(s )^2\Xg(s)\Big) \\
	 &+ \mathfrak a_3(s)  \Big( \Xg^3 (s) -3\tCG(s )\Xg(s)\Big)
	 +  \mathfrak a_1(s)  \Xg(s )
+ \ERR{1}(s ) 
\Big) \, ds  + \Zg(t ) 
\end{equs}
in the second scaling regime, such that the following holds.
For every $ T>0$ and $\ka>0$, there exists $C = C(T,\ka,\al)$ such that for all $0 \leq s \leq T$, $x\in\T^2$ and sufficiently small $ \ga >0$
\begin{equs} [e:Final_Error1]
|\ERR{1}(s, x)|
 & \leq \; C \; \ga^{-30 \al - \ka}
 	\big(  \| \Xg(s, \cdot) \|_{\Ca}^7 +1 \big)  \\
\times & \Big(  \ga^{\frac23} s^{-\frac13}   
	+ \| \Xg^{\mathrm{high}}(s, \cdot)  \|_{L^{\infty}(\T^2)} 
	+\|  Q_{\ga,s}(s,\cdot) \|_{L^\infty(\Le)} 
	+ |\tilde E(s,x)|   \Big) \;,
\end{equs}
where $\tilde E$ is defined in \eqref{e:averageA}.
Here $\ERR{1}$  is different in the two regimes but the bound holds for both regimes.
\end{lemma}

\begin{remark}
Recall the stopping time $\taun $ defined in \eqref{e:deftaug}. 
Denote by $X_{\ga,\nn}$ the solution to \eqref{e:evolution3CC} with $Z_\ga$ replaced by $Z_{\ga,\nn}$ 
and $\ERR{1}$ replaced by $\ERR{1}_{\nn}$  which is equal to $\ERR{1}$ before the time $\taun $ and is set to $0$ after $\taun $.
Taking the $L^p(\T^2)$ norm on both sides of \eqref{e:Final_Error1},
one has the bound
\begin{equs} [e:Final_Error1a]
 \|   \ERR{1}_{\nn} (s,\cdot)\|_{L^p(\T^2)}  
  \leq  & C \ga^{-(30 \al + \ka)}
\Big(  \ga^{\frac{2}{3}} s^{-\frac{1}{3}}   
+ \| \Xg^{\mathrm{high}} (s,\cdot)  \|_{L^{\infty}(\T^2)} \\
&+  \|  Q_{\ga,s}(s,\cdot) \|_{L^\infty(\Le)}  
+ \|\tilde E(s,\cdot)\|_{L^p(\T^2 )}  \Big) \;,
\end{equs}
where $C$ depends on $T,\nn,p,\ka,\nu$.
\end{remark}

\begin{proof}[of Lemma~\ref{lem:Error1}]
We first consider the second regime.
With the choice of parameters as in \eqref{e:tune-tri}, or equivalently \eqref{e:tune-tri-1} and \eqref{e:tune-tri-2}, the  discrete evolution  \eqref{e:mildBEG6}
can be written as
\begin{equs}  
\Xg(t,\cdot) =&\Pg{t} \Xng + \int_0^t \Pg{t-s}  \Kg \star  
	 \Big( C_{\beta,\theta}\Xg^5  (s,\cdot) +\Big(\frac92 \CGG + \mathfrak a_3\Big) \Xg^3  (s,\cdot) \\
+ & \Big( -3\CGG \mathfrak{a}_3 - \frac{27}{4} \CGG^2+ \mathfrak{a}_1 \Big) \Xg(s,\cdot)
 + E_\gamma(s,\cdot) \Big) \, ds  + Z_\gamma (s,\cdot)  
 \qquad \mbox{on } \Lambda_\eps.
\end{equs}
We apply $\Ex$ on both sides, and  
compare it with the continuous equation \eqref{e:evolution3CC}. We then have
\begin{equ} [e:many-errors]
\ERR{1} 
= \err{1} + \err{2} + \err{3}\;,
\end{equ}
where the error terms are given by
\begin{align*}
\err{1}(s ) &= \Eg(s ) +  \Big(C_{\beta,\theta} +\frac{9}{20} \Big)    \Ex (\Xg^5(s )) \;, \\
\err{2}(s )& 
	= - \frac{9}{20} \Big(  \Ex \; \big( \Xg^5(s ) \big) - (\Ex \; \Xg(s ))^5 \Big) \\
	&\quad +\Big(\frac92 \CGG + \mathfrak a_3 \Big) \Big(  \Ex \; \big( \Xg^3(s ) \big) - (\Ex \; \Xg(s ))^3 \Big)
	\;,\\
\err{3}(s )  &=   \Big(\frac92 \CGG +\mathfrak a_3 - \frac92 \tCG(s ) - \mathfrak a_3 (s) \Big) \;  \Xg^3(s)  \\
	& -   \Big(\frac{27}{4} \CGG^2 +3\mathfrak a_3\CGG -\mathfrak a_1
		-\frac{27}{4} \CGG(s)^2 -3\mathfrak a_3(s)\CGG(s) +\mathfrak a_1(s)\Big) \;  \Xg(s)\;, 
\end{align*}
where in the expression of $\err{3}$ and also below we simply denote $X_\gamma = \Ex X_\gamma$.
The analysis for $\err{1}$ and $\err{2}$ follow essentially the same way as in \cite[Proof of Lemma~7.1]{MourratWeber}, so we will only write down the bounds we eventually obtain for these errors.

For the first term $\err{1}$,
using the assumption \eqref{e:tune-tri}  on $(\be,\theta)$, and
the definition of $C_{\beta,\theta}$,
one has $|C_{\beta,\theta} +\frac{9}{20}| \le C\gamma^2 \CGG$.
Then by  the definition of $\Eg$ in \eqref{e:def-Eg}, and that $\CGG$ has only logarithmic divergence, we  can finally get that
for any arbitrary small $\ka >0$
\begin{align}
\big\| \err{1} (s,\cdot) \big\|_{L^\infty(\T^2)} &\leq  C(\ka,\al) \ga^{2-\ka - 30 \al}  \big( \|\Xg(s,\cdot) \|_{\Ca}^7  + 1 \big)\;.  \notag
\end{align}

For the second term $\err{2}$,  by decomposing $X_\gamma$
into low and high modes as in \eqref{e:defXlh}, we can obtain  the bound
\begin{align}
\| \err{2}(s, \cdot) \|_{L^\infty(\T^2)} 
\leq C(\ka) \ga^{-\ka- 15\al} \| \Xg^{\mathrm{high}}(s, \cdot)  \|_{L^{\infty}(\T^2)} 
 \| \Xg(s, \cdot) \|_{\Ca}^4  \;.
\end{align}

In order to control the  term $\err{3}$, 
we first consider the quantity
\begin{equ} [e:cccc]
\CGG -\CGG(s,x) + \lim_{\eps\to 0} (\Ce(s) - \Ce) \;,
\end{equ}
which is called $\CGG -\CGG(s,x) +A-A(s)$ in \cite[Proof of Lemma~7.1]{MourratWeber} (see the definition of $A_\eps(s)$ below \cite[(3.11)]{MourratWeber}); note that the $\eps\to 0$ limit is well-defined as discussed around \eqref{e:ce-cet} in the proof of Lemma~\ref{lem:ak-ake}.
By the definition  of $\CGG$   in \eqref{e:valueCGG},
the definition of $\CGG(s,x)$ in \eqref{e:def-cga-s}, and \eqref{e:ce-cet},
we have that for $x \in \T^2$, \eqref{e:cccc} is equal to
\begin{equ}
 \sum_{\substack{\om \in \{-N, \ldots, N \}^2\\ \om \neq 0 }} 
 	\frac{|\hKg(\om)|^2}{4 \beta_c \ga^{-b}  (1 - \hKg(\om))} 
 - [ R_{\ga,s}(\cdot,x)]_{ s}
 + \frac{s}{2\beta_c} 
 -  \sum_{\substack{\om \in \Z^2  \\ \om \neq 0} }  
 	\frac{\exp(-2s \pi^2|\om|^2 )}{4 \beta_c \pi^2 |\om|^2} \;.
\end{equ}
Here $b=4$ and $\beta_c =3$ since we are considering the second regime. 
Recall from \eqref{e:def:Qg} that for $x\in \Le$,
$[ R_{\ga,r}(\cdot,x)]_{ r}  = \langle R_{\ga,r}(\cdot,x) \rangle_r\ + Q_{\ga,r}(s,x)$.
According to \eqref{e:QuadrVarR} we get for $x\in \Le$
\begin{equs}
\langle& R_{\ga,s}(\cdot, x) \rangle_s \\
 &= \co^2 \int_0^s \sum_{z \in \Le} \eps^2  \big(\Pg{s-r} \ae \Kg \big)^2 (x-z) \!\!\!\sum_{\bar{\sigma} \in \{ \pm 1,0\} } (\bar{\sigma} - \sigma(\alpha^{-1}r,\eps^{-1} z))^2 \Cgn\big(r,  z, \bar{\sigma}) \, dr   \\
&=  \frac{2}{\beta_c}  \int_0^s \sum_{z \in \Le} \eps^2  \big(\Pg{s-r} \ae \Kg \big)^2 (x-z) \;  \, dr  +\err{4}(s,x) + \tilde E_s(s,x) \\
&= \frac{1}{2\beta_c} \int_0^s  \!\!\! 
	\sum_{\om \in \{-N, \ldots, N\}^2}
	\exp\Big(-\frac{2 r}{\ga^b}\big(1 -\hKg(\om)\big) \Big)  \,  \big|  \hKg (\om)\big|^2\, dr +\err{4}(s,x) + \tilde E_s(s,x)\\
&=   \frac{s}{2\beta_c}
	+  \!\!\!\!\!\! \sum_{\substack{\om \in \{-N, \ldots, N\}^2\\\om \neq 0}}  \frac{\big|  \hKg (\om)\big|^2 }{4 \beta_c \ga^{-b}\big(1 -\hKg(\om)\big)} 
		\Big( 1- e^{-\frac{2 s}{\ga^b}\Ll(1 -\hKg(\om)\Rr) } \Big)    +\err{4}(s,x) + \tilde E_s(s,x)
\end{equs}
where $\err{4}$ is the error that arises by replacing $\co^2$ in the second line by $1$, and
 $\tilde E$ is defined in \eqref{e:averageA}.
By $|\co^2 -1|\le \ga^2$ and $
\int_0^s   \sum_{z \in \Le}  \eg^2 \,  \big( \Pg{t-r} \ae \Kg \big)^2(z-x)
 	   \, dr 
	 \leq C \log \gamma^{-1} \le C(\ka) \ga^{-\ka}$ one has 
$|\err{4}(s,x)|\le C\ga^{2-\ka}$. 
Proposition~\ref{prop:averageA} gives the stochastic bound on 
$\tilde E_s(s,x)$.
	 

Therefore up to  the terms $Q_{\ga,s}(s,x)$, $ \err{4}(s,x) $ and 
$\tilde E_s(s,x)$,  the quantity
\eqref{e:cccc} is equal to
\begin{equ} \label{e:ugly0}
 \sum_{\substack{\om \in \{-N, \ldots, N\}^2 \\ \om \neq 0}}  
  \frac{\big|  \hKg (\om)\big|^2 }{4 \beta_c \ga^{-b} \big(1 -\hKg(\om)\big)}   e^{- \frac{2 s }{\ga^b }\Ll(1 -\hKg(\om)\Rr)} 
 -  \sum_{\substack{\om \in \Z^2  \\ \om \neq 0} }  
 \frac{\exp(-2s \pi^2|\om|^2 )}{4\beta_c \pi^2 |\om|^2} \;.
\end{equ}
We bound the sums over $|\om| < \ga^{-2}$ and $|\om| \ge \ga^{-2}$  separately. In the
case $|\om| < \ga^{-2}$ we use the fact that according to Lemma ~\ref{le:Kg0} $\gamma^{-4}(1-\hKg(\om))$ 
approximates $\pi^2|\om|^2$ up to an error $\leq C\gamma^2 |\om|^3$ (which implies 
in particular that $\hKg(\om)$ approximates $1$ up to an error $\leq C \gamma^{4}|\om|^2$).
For  $|\om| \ge \ga^{-2}$ we treat the two sums separately and use Lemma~\ref{le:Kg}
which yields in particular the upper bound $|\hKg(\om)| \leq C |\gamma^2 \om |^{-2}$
as well as the lower bound $1 - \Kg(\om) \geq 1$. After some calculations (the details of 
which are as in \cite[Equation (7.7) ]{MourratWeber}) we conclude that  \eqref{e:ugly0}
is bounded by $C\ga^{\frac23} s^{-\frac{1}{3}}$.

%
%
%
%

Now to really bound the coefficients appearing in  $ \err{3}(s,x) $, 
note that the coefficient of $ \Xg^3(s) $ in $ \err{3}(s,x) $ can be expressed as
\[
 \frac92 \CGG +\mathfrak a_3 - \frac92 \tCG(s ) - \mathfrak a_3 (s) 
 =\frac92 (\CGG -\tCG(s ) + \bar{\mathfrak c}(s))
 \]
 which is exactly the quantity \eqref{e:cccc} we have bounded times $\frac92$.
Furthermore, the absolute value of the coefficient of $ \Xg (s) $ in $ \err{3}(s,x) $ is 
\begin{equs}
\Big| \frac{27}{4} \CGG^2 &+3\mathfrak a_3\CGG -\mathfrak a_1
		-\frac{27}{4} \CGG(s)^2 -3\mathfrak a_3(s)\CGG(s) +\mathfrak a_1(s) \Big| \\
&= \Big| \frac{27}{4} \Big( \CGG^2  - \CGG(s)^2 \Big) 
	+3 \Big(\mathfrak a_3 \CGG  +\frac92 \bar{\mathfrak c} (s) \tCG(s) -\mathfrak a_3 \CGG(s) \Big) + 3\mathfrak a_3 \bar{\mathfrak  c}(s)  
	-\frac{27}{4} \bar{\mathfrak c}(s)^2 \Big| \\
& =\frac{27}{4} \Big| \CGG(s)+\CGG -\bar{\mathfrak  c}(s)  +\frac49 \mathfrak a_3 \Big| \cdot
	\Big|\CGG - \CGG(s)+\bar{\mathfrak  c}(s) \Big| \\
&\le C(\ka) \ga^{-\kappa} \big| \CGG - \CGG(s)+\bar{\mathfrak  c}(s) \big|
\end{equs}
where in the second line we applied \eqref{e:a3a3a1a1},   the third line is obtained by elementary factorization, and in the last line we  have used that each term
in  $ \CGG(s)+\CGG -\bar{\mathfrak  c}(s)  +\frac49 \mathfrak a_3$ is bounded by $\leq C \log \gamma^{-1}$ uniformly in $s$.
So the bound of this coefficient again boils down to the bound on \eqref{e:cccc}.

%

The $\tilde E$ dependent terms in $\err{3}$ are
\begin{equ} [e:def-Errstar]
-\frac92 \tilde E_s(s,x) X_\ga^3(s,x) + 
\frac{27}{4} \Big( \CGG(s)+\CGG -\bar{\mathfrak  c}(s)  +\frac49 \mathfrak a_3 \Big)\tilde E_s(s,x) X_\ga(s,x)
\end{equ}
%
whose absolute value is bounded by
\[
C(\nu,\ka) \ga^{-10\nu-\ka} \Big( \|X_\ga(s,\cdot)\|_{\mathcal C^{-\nu}}^3+1 \Big) |\tilde E_s(s,x)|  \;.
\]
Summarizing all the above bounds we obtain \eqref{e:Final_Error1}.

The proof for the first regime is analogous and is thus omitted; in particular we can obtain bounds with slightly larger (but still negative) powers of $\ga$ and lower powers of $ \| \Xg(s, \cdot) \|_{\Ca}$ than that in \eqref{e:Final_Error1} 
but the latter is sufficient for our purpose.
\end{proof}

The error (5) is bounded by \cite[Lemma~7.3]{MourratWeber} as
\begin{equ} [e:error5]
\sup_{0 \leq t \leq T} \|  P_t^\gamma X_\gamma^0 - P_t X^0  \|_{\Ca} 
 \leq C \| \Xn - \Xng \|_{\Ca} + \bar C \ga^{\frac{\ka}{2}}
\to 0
\end{equ} 
for every $T>0$, where $\bar C$ depends on $\al, \ka,T$ and $\| \Xng \|_{\Cc^{-\al + \ka}}$.

In the sequel, we let $\bar n=3$ in the first regime and  $\bar n=5$ in the second regime.

At this stage, note that if we define
\begin{equ} [e:defXbar]
\oX(t, \cdot) \eqdef 
	P_{t}\Xn +  \Zgn(t, \cdot)   
	+ \Ss_T(\Zgn, \ZGn{2},\cdots,\ZGn{\bar n})(t, \cdot)\;,
\end{equ}
where $\Ss_T$ is the solution map defined in Theorem~\ref{thm:ContSolution},
then by the convergence in law of 
$(\Zgn, \ZGn{2}, \cdots ,\ZGn{\bar n})$  with respect to the topology of $L^\infty([0,T],\Cc^{-\al})^{\bar n}$ 
to $(Z, Z^{: 2 :},\cdots,Z^{: \bar n:})$,
and by the continuity of the map  $\Ss_T$ as stated in Theorem~\ref{thm:ContSolution}, one has that $\oX$ converges in law to $X$. 

Therefore, it remains to compare $\oX$ and $\tXgn$. The idea is to follow a discrete version of Da Prato-Debussche argument \cite{dPD}, namely, setting
\begin{equs} [e:Defvgm]
\tvg(t,x)  &\eqdef \tXgn(t,x)  -  \Zgn(t,x) - P_t^\gamma X_\gamma^0 (t,x) \qquad   x \in \T^2 \;, \\
\bvg(t, x) &\eqdef \oX(t, x)  -  \Zgn(t, x) -  P_t X^0(t, x)   \qquad  x \in \T^2    \;,
\end{equs}
and we compare $\tvg$ and $\bvg$. 
Define
\begin{equ} [e:tZkbZk]
\ttZ^{: k:} \eqdef \sum_{\ell=0}^k  
	 (P_t^\gamma X_\gamma^0 )^{k-\ell}  \Zgn^{:\ell:} \;,
\qquad
\bZZg^{: k:} \eqdef \sum_{\ell=0}^k  
	 (P_t X^0 )^{k-\ell}  \Zgn^{:\ell:} \;.
 \end{equ}
Note that if the above Wick powers were defined via Hermite polynomials rather than iterated integrals then the above identities would follow from basic properties of Hermite polynomials $H_k(x+y) = \sum_{\ell=0}^k  x^\ell H_{k-\ell}(y)$.

Now it is straightforward to check that $\bvg $ satisfies
\begin{align} \label{e:evolution3aa}
 \bvg(t)  = - \int_0^t P_{t-s}   \bPsg( s) \, ds\; , 
\end{align}
where we have set 
\begin{equ} [e:Def_Psi_bar-1]
\bPsg(s)   \eqdef   \frac{a_c(4a_c-1) \beta_c^3 }{3(2a_c+1)^2}   \sum_{k=0}^3 {3 \choose k} \bZZg^{: k:}(s)\,\bvg^{3-k}(s) 
- \mathfrak a_1(s)\,\big(\bvg(s) + \bZZg(s) \big) 
\end{equ}
in the first regime and
\begin{equs} [e:Def_Psi_bar]
\bPsg(s) &  \eqdef     \frac{9}{20} \sum_{k=0}^5 {5 \choose k} \bZZg^{: k:}(s)\,\bvg^{5-k}(s) 
-\mathfrak a_3(s) \sum_{k=0}^3 {3 \choose k} \bZZg^{: k:}(s)\,\bvg^{3-k}(s) \\
& \qquad
- \mathfrak a_1(s)\,\big(\bvg(s) + \bZZg(s) \big) \;.
\end{equs}
in the second regime.
On the other hand, by Lemma~\ref{lem:Error1} and \eqref{e:Defvgm}, $\tvg$ satisfies (on $\T^2$)
\begin{align} \label{e:evolution3b} 
\tvg(t) =&  - \int_0^t \Pg{t-s}  \Kg \star   \big( \Psg(s)
+ \ERR{1}_{\nn}   
+\ERR{2}_{\nn}(s,\cdot) \big) \, ds \;,   
\end{align}
where
$\Psg(s)$ is defined in the same way as \eqref{e:Def_Psi_bar-1} or \eqref{e:Def_Psi_bar} with $\bZZg^{: k:}, \bZZg$ 
replaced by $\ttZ^{: k:}, \ttZ$
and $\bvg$ replaced by $\tvg$.
Here the term $ \ERR{1}_{\nn}$ was estimated in Lemma~\ref{lem:Error1},
and $ \ERR{2}_{\nn}$  controls the error (6) i.e. the fact that the iterated integrals 
do not exactly coincide with Hermite Polynomials. In fact, the difference between 
Hermite polynomials and iterated integrals was already bounded  in Lemma~\ref{p:discrete-Wick}.
Relying on these bounds and using \eqref{e:tZkbZk} it is straightforward 
(see \cite[Lemma~7.4]{MourratWeber} for the analogous details in the Kac-Ising case)
to check  that 
in both regimes one has for $0 \leq s \leq T$ 
\begin{equ} [e:e2Bound]
\| \ERR{2}_{\nn}(s, \cdot) \|_{L^\infty(\T^2)} 
 \leq C(T, \al, \ka) \Big( 1+s^{-3\nu-\kappa} + \|\tvg\|_{\Cc^{\frac12}}^4\Big)
 \sum_{k=2}^5 \| E_{\ga,s,\nn}^{: k : }(s,\cdot)\|_{L^\infty(\T^2)}  
\end{equ}
where $\EG{n}(s,x)$
was introduced in Proposition~\ref{p:discrete-Wick}.
The following estimate holds in both regimes.

\begin{lemma} \label{lem:preGronwall}
 For every $0 \leq t \leq T$ and sufficiently small $ \ga >0$, we have
\begin{align}
\|  \bvg(t, \cdot)-  \tvg(t, \cdot)  \big\|_{\Cc^{\frac12}} \leq  & \overline{C}_1  \int_0^t  (t-s)^{-\frac{1}{3}} s^{-\frac{1}{6} }\|  \bvg(s,\cdot) - \tvg(s,\cdot) \|_{\Cc^{\frac12}}   \, ds \notag\\
&+ \overline{C}_1 (\ga^{\frac{\ka}{2}}  + \| \Xng - \Xn \|_{\Ca}) + \ERR{3}(t)  \label{e:Gronwall1} \;,
\end{align}
where the constant $\overline{C}_1$ depends on $\nu,\ka,T$, $\| \Xn \|_{\Cc^{- \al + \ka}}$, $\| \Xng \|_{\Cc^{-\al + \ka}}$ as well  as 
the random quantities 
$\sup_{0 \leq s \leq T} \| \bvg(s, \cdot) \|_{\Cc^{\frac12}}$,   
$\sup_{0 \leq s \leq T}\|\tvg(s, \cdot) \|_{\Cc^{\frac12}}$,  
and  
\[
\sup_{0 \leq s \leq T} \| \ZGn{k}(s, \cdot) \|_{\Ca} \quad \mbox{for } k=1,\dots, \bar n \;.
\]

There exists some $p \geq 2$, such that the error term $\ERR{3}$ satisfies that for every $T \geq 0$ and $0 <\la\le \frac{1}{2}$
\begin{align}\label{e:very_good_bound_A}
\E \sup_{0 \leq t \leq T}  \big| \ERR{3}(t) \big|^p \leq \overline{C}_2 \ga^{\la } \;,
\end{align}
for a constant $\overline{C}_2= \overline{C}_2(p,T,\la)$.
\end{lemma}

\begin{proof}
Using \eqref{e:evolution3aa} - \eqref{e:evolution3b}, we get that
for any $ t \geq 0$ and $\gamma >0$, 
\begin{equs} [e:bound_diff1]
 \bvg(t, \cdot) & -  \tvg(t, \cdot)  = -\int_0^t \big(  P_{t-s} - \Pg{t-s} \star \Kg   \big)  \,\bPsg\big(  s\big) \,ds \notag  \\
&- \int_0^t \Pg{t-s}  \star \Kg \star\, \big( \bPsg( s) -  \Psg(s)    \big) \, ds\; \notag\\
&+ \int_0^t \Pg{t-s}\star  \,\Kg \star 
	\big(\ERR{1}_{\nn}(s, \cdot) 
	+\ERR{2}_{\nn}(s,\cdot) \big)   \, ds\;, 
\end{equs}
where $\bPsg( s) $ was defined in \eqref{e:Def_Psi_bar} and   $\Psg(s)$ was defined below  \eqref{e:evolution3b}.  
The rest of the proof relies on the crucial multiplicative inequality \cite[Lemma A.5]{MourratWeber} which is the linchpin around which the Da Prato-Debussche argument revolves (see \cite[Proposition~2.1]{dPD} for a similar result); it
states that 
if $\beta < 0 < \al$ with $\al + \beta >0$, then  there exists a constant $C$ depending only on $\al$ and $\beta$ such that 
\begin{equ} [e:multip-ineq]
\| Z_1 \, Z_2 \|_{\Cc^\beta} \leq C \|Z_1 \|_{\Cc^\al} \, \| Z_2\|_{\Cc^\beta}\;. 
\end{equ}
Proceeding as in the proof of \cite[Lemma~7.5]{MourratWeber}, which uses the above multiplicative inequality, together with the (discrete) heat kernel estimates in Sec.~8 of that reference,  
we can bound $\| \bPsg( s)\|_{\Ca}$ in \eqref{e:bound_diff1} in terms of $\|\bvg(s, \cdot) \|_{\Cc^\frac12}$ and $\| \bZZg^{:k:}(s, \cdot)\|_{\Ca}$ where $\al < \frac12$,
and the latter quantity is by \eqref{e:tZkbZk} further bounded in terms of $\|  \Zgn^{:k:}(s, \cdot) \|_{\Ca}$ and $\|\Xn \|_{\Ca}$.
Therefore the $\Cc^{\frac12}$ norm of the first term on the RHS of \eqref{e:bound_diff1} can be eventually bounded by
$C\ga^{\frac12}$ where $C$ may depend on all the quantities stated in the lemma,
and the small factor $\ga^{\frac12}$ arises from a bound on $\big\| \big(P_{t} -\Pg{t} \star \Kg   \big) \big\|_{\Ca  \to \Cc^{\frac12}}$.

The $\Cc^{\frac12}$ norm of the second term  on the RHS of \eqref{e:bound_diff1} can be bounded in the same way using the multiplicative inequality \eqref{e:multip-ineq} and heat kernel estimates, by
\begin{equ}
C    \int_0^t  (t-s)^{-\frac13} s^{-\frac16 } \|  \bvg(s) - \tvg(s) \|_{\Cc^{\frac12}}\, ds  + C\|  \Xng  -\Xn \|_{\Ca}  + C \ga^{\frac{\ka}{2}} \;,
\end{equ}
where again $C$ may depend on all the quantities stated in the lemma.

Now we consider the $\Cc^{\frac12}$ norm of the last term on the RHS of \eqref{e:bound_diff1}. 
We use \cite[Rem.~3.6 and Prop.~3.7]{MourratWeberGlobal}
which state that the space $L^p$  is continuously embedded in $\mathcal B^0_{p,\infty}$ and the latter is further continuously embedded in $\mathcal B^\alpha_{\infty,\infty}$ (i.e. the space $\Cc^{\alpha}$) provided that $\alpha+2/p=0$. Thus applying \eqref{e:Final_Error1a}, we have that for any $\bar{\ka}>0$ there exists $C=C(p,\bar\ka)$ such that
\begin{equs} [e:int-time-Err1]
\Big\| \int_0^t  &   \Pg{t-s}\star  \,\Kg \star  \ERR{1}_{\nn}(s, \cdot) \, ds \Big\|_{\Cc^{\frac12}}
\le C \int_0^t (t-s)^{-\frac14-\frac{1}{p}-\bar\ka} 
	\big\| \ERR{1}_{\nn}(s, \cdot) \big\|_{L^p(\T^2)} \, ds  \\
&\le C \ga^{-(30 \al + \ka)}
 \int_0^t (t-s)^{-\frac14-\frac{1}{p}-\bar\ka}  \Big(  \ga^{\frac{2}{3}} s^{-\frac{1}{3}}   
+ \| \Xg^{\mathrm{high}} (s,\cdot)  \|_{L^{\infty}(\T^2)} \\
&\qquad +  \|  Q_{\ga,s}(s,\cdot) \|_{L^\infty(\Le)}  
+ \|\tilde E(s,\cdot)\|_{L^p(\T^2 )}  \Big) \,ds \\
&\le 
C\ga^{\frac23-(30 \al + \ka)} 
	\int_0^t (t-s)^{-\frac14-\frac{1}{p}-\bar\ka}  s^{-\frac13}ds \\
&\qquad 
+C \ga^{-(30 \al + \ka)}\Big(
	\int_0^t (t-s)^{- (\frac14+\frac{1}{p}+\bar\ka) p_\star}   \, ds 
	\Big)^{\frac{1}{p_\star}} 
\\
&\qquad\qquad \times
	\Big(  \| \Xg^{\mathrm{high}}  \|_{L^{\infty}(\T^2\times [0,T])}
	+ \|  Q_{\ga,s} \|_{L^\infty(\Le\times [0,T])}
	+\|\tilde E\|_{L^p(\T^2\times [0,T] )} \Big)
\end{equs}
where $p_\star$ is such that $\frac{1}{p_\star} + \frac{1}{p}=1$.
Choosing (and fixing from now on) $p$ sufficiently large (depending only on $\bar\ka$) the above expression 
can be bounded by 
\[
C(T,p,\bar\ka) \ga^{-(30 \al + \ka)} \Big( 
	\ga^{\frac23}+
	\| \Xg^{\mathrm{high}}  \|_{L^{\infty}(\T^2\times [0,T])}
	+ \|  Q_{\ga,s} \|_{L^\infty(\Le\times [0,T])}
	+\|\tilde E\|_{L^p(\T^2\times [0,T] )} 
	\Big) \;.
\]
We have Proposition~\ref{prop:averageA} to bound $\tilde E$, Proposition~\ref{prop:boundQ} to bound $Q_{\ga}$.
Regarding the term $\Xg^{\mathrm{high}}$,
which is equal to
$  \Zg^{\mathrm{high}}  +  \tvg^{\mathrm{high}}
+ (P_s^\ga X_\ga^0)^{\mathrm{high}}  $,
we can bound $\| \cdot \|_{L^{\infty}(\T^2)}$ of the last two quantities by $C \ga^{1}\| \cdot  \|_{\Cc^\frac12} $. 
Finally for  $\Zg^{\mathrm{high}}$, by \cite[Lemmas~4.6]{MourratWeber}  with minor changes in the proof
due to the scaling-regime-dependent definition \eqref{e:defXlh} and 
kernel estimates in Section~\ref{sec:kernels},
one has $\E \|\Xg^{\mathrm{high}}(s, \cdot)  \|_{L^{\infty}}^p \le C\ga^{p(1-\kappa)}$.
Therefore by choosing $\nu,\ka$ small enough depending on the previously fixed $p$ one has that 
\begin{equ}
\E \sup_{0 \leq t \leq T} 
\Big\| \int_0^t    \Pg{t-s}\star  \,\Kg \star  \ERR{1}_{\nn}(s, \cdot) \, ds \Big\|^p_{\Cc^{\frac12}}\le C(p,T) \ga^{\frac12} \;.
\end{equ}

Similarly for $\ERR{2}_\nn$, invoking Proposition~\ref{p:discrete-Wick} 
to bound $ E_{\ga,s,\nn}^{: k : }$, one has
\begin{equ}
\E \sup_{0 \leq t \leq T}  
\Big\| \int_0^t     \Pg{t-s}\star  \,\Kg \star  \ERR{2}_{\nn}(s, \cdot) \, ds \Big\|^p_{\Cc^{\frac{1}{2}}}
\le C(p,T) \ga^{\frac{p}{2}} \;.
\end{equ}
Therefore \eqref{e:very_good_bound_A} is obtained.
%
%
%
\end{proof}

Now we prove our main theorem of the article.

\begin{proof}[of Theorem~\ref{thm:Main}]
The proof is essentially the same as \cite{MourratWeber}; we give the proof for completeness.
Our arguments hold for both scaling regimes.
For $\rr$ and $\nn \ge 1$, we define the events $ \Aa_{\rr}^Z =  \Aa_{\rr}^Z(\ga,\nn)$,  and $ \Aa^{\mathsf{E}} =  \Aa^{\mathsf{E}}(\ga,\nn)$ by
\begin{align*}
 \Aa_{\rr}^Z \eqdef& \, \big\{   \|\ZGn{k}    \|_{\Ca}   \leq \rr \text{ on } [0,T], \;\,\; k=1,\dots,5\big\}\; , \\
 \Aa^{\mathsf{E}} \eqdef&  \big\{ \sup_{0 \leq t \leq T}  \big| \ERR{3}(t) \big| \leq  \gamma^{\frac{1}{2p}}  \,   \big\} \,,
\end{align*}
where $p$ is the constant in \eqref{e:very_good_bound_A}.
For every $\nn, \rr\geq 1$ and every bounded uniformly continuous mapping $F : \Dd([0,T] , \Ca) \to \R$, we have
\begin{equs} [e:final-converge]
\big| \E \big( & F (   \tXgn ) \big)  -  \E \big( F (X ) \big) \big| 
 \leq \big| \E \big( F ( \oX ) \big) - \E \big( F (X ) \big)  \big| \\  
 &  + \E   \Big(  \Ll| F\Ll( \oX \Rr) -  F\Ll( \tXgn \Rr)  \Rr| \mathbf{1}_{ \Aa^Z_{\rr} \cap \Aa^{\mathsf{E}}}  \Big)  
+ \| F\|_{L^\infty} \, \P \Big( \ov{\Aa}^Z_{\rr} \cup \ov{\Aa}^\mathsf{E} \Big) \;.
\end{equs}
Recall that $\oX$ converges in law to $X$, see \eqref{e:defXbar} and the discussion below  it.
 
To bound the second term on RHS of  \eqref{e:final-converge},
note that on the event $\Aa_{\rr}^Z $ and by  continuity of $\Ss_T$ (Theorem~\ref{thm:ContSolution}), we have $\sup_{0 \leq t \leq  T } \|  \bvg(t)\|_{\Cc^{\frac12}} \le C(T,\rr)$ for some finite constant $C(T,\rr)$.
Applying Gronwall's inequality to the bound obtained in Lemma~\ref{lem:preGronwall}, one has that
on the event $\Aa_{\rr}^Z \cap \Aa^{\mathsf{E}}$
\begin{equ} \label{e:Gronwall452}
\Ll\|  \tvg(t, \cdot)-  \bvg(t, \cdot)  \Rr\|_{\Cc^{\frac12}} \leq   C \Ll(\ga^{\frac{\ka}{2}}  + \| \Xng - \Xn \|_{\Ca} \Rr)
\end{equ} 
for all $t\ge 0$ such that 
$\|\tvg(t)\|_{\Cc^{\frac12}} \le C(T,\rr)+2$.
%
In particular for $\ga$ small enough, the right hand side of \eqref{e:Gronwall452}
is bounded by $1$. By continuity of $v_\ga$ and $\bar v_\ga$ (which follows by definition \eqref{e:Defvgm} - the jumps in the evolution of $\Xg$ are all contained in the part $\Zgn$), the bound \eqref{e:Gronwall452} 
must actually hold for all $t\in[0,T]$.

This together with \eqref{e:error5}, \eqref{e:Defvgm} implies
that the second term on RHS of  \eqref{e:final-converge} vanishes.

Regarding the last term in  \eqref{e:final-converge},
 it follows from \eqref{e:very_good_bound_A} i.e. the bound for $\ERR{3}(t)$ and Chebyshev's inequality that
$\lim_{\ga \to 0} \P ( \Aa^\mathsf{E}) = 1 $.
For the event $ \Aa_{\rr}^Z$,
we know that the limiting quantities $\sup_{0 \le t \le T}  \| Z^{:k:}(t) \|_{\Ca}$ are finite a.s.; on the other hand it is easy to argue that 
the stopping time that $ \|\ZGn{k} (t)  \|_{\Ca}  $ first exceeds the value $\rr$
will converge to  \footnote{outside a countable set of $\rr$ that $ \|\ZGn{k} (t)  \|_{\Ca}  $ attains $\rr$ as a local maximum with positive probability}
the stopping time that $\| Z^{:k:}(t) \|_{\Ca}$ first exceeds the same value $\rr$.
Thus we can choose $\rr$  large enough, so that
$\liminf_{\ga \to 0} \P (\Aa^Z_{\rr} )$ is arbitrarily close to $1$.

This proves that $\tXgn$ converges in law to $X$ as $\ga$ tends to~$0$, for any fixed value of $\nn$. We can remove $\nn$ by the same reasoning as above.
The stopping time  $\taun$ defined in \eqref{e:deftaug} converges in law to 
\footnote{outside a countable set for the same reason}
 the stopping time $\tau_\nn$ defined in the same way for $X$, for every $\nn$. Moreover, we know from Theorem~\ref{thm:ContSolution} that $\sup_{0 \leq t \leq T+1} \| X(t) \|_{\Ca}$ is a.s. finite. Hence by choosing $\nn = \nn(T,\eps)$ sufficiently large, 
 $ \liminf_{\ga \to 0} \P ( \tXgn = \Xg )$  can be made arbitrarily close to $1$.
Therefore we have proved that $X_\ga$ also converges in law to $X$.

This concludes the proof of Theorem~\ref{thm:Main}.
Note that item (2) of the theorem is clearly just the degenerate case 
of the item (1)
that the cubic term equals zero and therefore one obtains a linear limit.
\end{proof}

\section{Appendix: Kernel estimates} \label{sec:kernels}

We need some estimates about $K_\ga$ and $P^\ga$. In the case of the first scaling regime \eqref{e:Phi4-scaling}, these estimates are proved in \cite[Section~8]{MourratWeber}. For the second scaling regime \eqref{e:Phi6-scaling}, we list all these results, without proving them since the proofs follow exactly the same way except that one simply applies the new scaling relations. 

We  begin with the Fourier transforms of these kernels. For  $\om \in \{-N, \ldots, N\}^2$, 
\begin{align}
\hKg (\om) = 
 \sum_{x\in \Le} \eg^2 \, \Kg(x) \,e^{- i \pi \om \cdot x}  =  \,
\ct\sum_{x\in \ga \Z^2_\star} \ga^{2} \, \KK( x) \,e^{- i \pi (\eps/ \ga) \om \cdot x}\;, \label{e:B1}
\end{align} 
where $\KK$ is the smooth function introduced in \eqref{e:norm-kk}, $\ga \Z^2_{\star} \eqdef \ga \Z^2 \setminus \{0\}$, and note that  $\eps/ \ga \approx \ga$ in the first regime and $\eps/ \ga \approx \ga^2$ in the second regime. Also,
\begin{equation}
\label{e:Fourier-semi}
\hPg{t}(\om) = \exp\Big(t \gamma^{-b} (\hKg(\om) - 1) \Big) \;, 
\end{equation} 
where $b=2$ in the first regime and $b=4$ in the second regime.

We now list some estimates which state that some properties of 
$\hKK(\ga \om)$ (resp. $\hKK(\ga^2 \om)$) also hold for $\hKg$ in the first (resp. second) regime, uniformly in $\ga$. 

\begin{lemma}\label{le:Kg0}
The following statement holds with $b=1$ in the first regime and $b=2$ in the second regime. There exists $C>0$ such that for all $0 < \ga< \frac13$ and for $|\om| \leq \ga^{-b}$ we have for $j=1,2$
\begin{align}
\big| \ga^{-2b} (1 -\hKg(\om) ) -  \pi^2 |\om|^2 \big| 
	&\leq C \ga^b | \om|^3\;, \label{e:K2.4} \\
\big|  -\ga^{-2b}  \partial_j \hKg(\om)  - 2  \pi^2 \om_j \big| 
	&\leq C \ga^b | \om|^2\;,\label{e:K2.3}\\
\big|  -\ga^{-2b}  \partial_j^2 \hKg(\om)  - 2  \pi^2 \big| 
	&\leq C \ga^b | \om| \;.\label{e:K2.2}
\end{align}
\end{lemma}

\begin{lemma}
\label{le:Kg}
The following statements hold with $b=1$ in the first regime and $b=2$ in the second regime.
There exists $C>0$ such that for all $0<\ga<\frac13$, $\om \in [-N-\frac12, N +\frac12]^2$ and $j=1,2$,

(1) (Estimates most useful for $|\om| \leq \ga^{-b}$)
\begin{align}\label{e:K1}
|\hKg(\om) | \leq 1 \;, \quad  
|\partial_j \hKg(\om) | \leq C \ga^b \big(  |\ga^b\om| \wedge 1  \big)  \;, \quad 
|\partial_j^2 \hKg(\om) | \leq  C \ga^{2b} \;. 
\end{align}

(2) (Estimates most useful for $|\om| \geq \ga^{-b}$)
\begin{align}
    | \ga^b \om|^2 \;  \big| \hKg (\om) \big| \leq C , \;\;\; 
 |  \ga^b \om|^2 \;  \big|  \partial_j \hKg (\om) \big| \leq C \ga^{b} , \;\;\; 
  |  \ga^b \om|^2 \;  \big| \partial_j^2\hKg (\om) \big| \leq C \ga^{2b} . \label{e:K3B}
\end{align}
Furthermore, there exist constants $C_1>0$ and $\ga_0 >0$ such that for all $0 <  \gamma < \ga_0$ and  $\om \in [-N-\frac12, N +\frac12]^2$ ,
\begin{equ}
1 -\hKg(\om)  \geq \frac{1}{C_1}  \big( |\ga^b \om|^2  \wedge 1 \big)\;.\label{e:K2} 
\end{equ} 
\end{lemma}

\begin{lemma}\label{le:Pgt}
Let $\ga_0>0$ be the constant introduced in Lemma~\ref{le:Kg}. For every $T>0$, there exists a constant $C = C(T)$ such that for all $0 < \gamma<\ga_0$, $0\leq t \leq T$ and $x \in \T^2$, we have
\begin{equation}\label{e:P0}
\big|    \Pg{t} \star  \Kg(x) \big| 
	\leq C \big( t^{-1} \big(\log(\ga^{-1})\big)^2    \wedge \gamma^{-2b} \log(\ga^{-1}) \big) \;,
\end{equation}
where $b=1$ in the first regime and $b=2$ in the second regime.
\end{lemma}


\bibliographystyle{./Martin}
\bibliography{./refs}

\end{document}